\documentclass{article}
\usepackage{amsmath}
\usepackage{amssymb}
\usepackage{amsfonts}
\usepackage{xcolor}
\usepackage{amsthm}
\usepackage{graphicx} % Required for inserting images
\usepackage{subcaption}
\usepackage[hidelinks]{hyperref}
\usepackage{physics}
\usepackage[a4paper]{geometry}
\usepackage{physics}
\usepackage{authblk}
\usepackage{booktabs}
\usepackage{lmodern}

\newtheorem{definition}{Definition}[section]
\newtheorem{remark}{Remark}[section]
\newtheorem{theorem}{Theorem}[section]
\newtheorem{lemma}{Lemma}[section]
\newtheorem{corollary}{Corollary}[section]
\newtheorem{assumption}{Assumption}[section]

\numberwithin{equation}{section}

\newcommand{\N}{\mathbb{N}}
\newcommand{\R}{\mathbb{R}}
\newcommand{\Ac}{\mathcal{A}}

\newcommand{\cC}{\mathcal{C}}

\title{Approximation properties of neural ODEs}
\author{Arturo De Marinis\footnote{\tt\small Gran Sasso Science Institute (GSSI), L'Aquila, Italy, \{arturo.demarinis, nicola.guglielmi, francesco.tudisco\}[at]gssi.it}, Davide Murari\footnote{\tt\small University of Cambridge, Cambridge, United Kingdom, dm2011[at]cam.ac.uk}, Elena Celledoni\footnote{\tt\small Norwegian University of Science and Technology (NTNU), Trondheim, Norway, \{elena.celledoni, brynjulf.owren\}[at]ntnu.no}, Nicola Guglielmi$^*$, Brynjulf Owren$^\ddagger$, Francesco Tudisco$^*$\footnote{\tt\small University of Edinburgh, Edinburgh, United Kingdom, f.tudisco[at]ed.ac.uk }}
\date{}

\begin{document}

\maketitle

\begin{abstract}
    We study the approximation properties of neural ordinary differential equations (neural ODEs) in the space of continuous functions. Since a neural ODE requires input and output dimensions to be the same, while input and output dimensions of a continuous function are generally different, we need to embed an input into the latent space of the neural ODE, and to project the output of the neural ODE into the output space. By composing the neural ODE flow map with such embedding and projection operations, we get a shallow neural network whose activation function is defined as the flow map of the neural ODE at the final time of the integration interval. Thus, the study of the approximation properties of neural ODEs leads to the study of the approximation properties of shallow neural networks with a particular choice of activation function. We prove the universal approximation property (UAP) of such shallow neural networks in the space of continuous functions. Furthermore, we investigate the approximation properties of shallow neural networks whose parameters satisfy specific constraints. In particular, we constrain the Lipschitz constant of the neural ODE's flow map and the norms of the weights to increase the network's stability. We prove that the UAP holds if we consider either constraint independently. When both are enforced, there is a loss of expressiveness, and we derive approximation bounds that quantify how accurately such a constrained network can approximate a continuous function.
\end{abstract}

\section{Introduction}
Neural ordinary differential equations (neural ODEs) \cite{chen2018neural,haber2017stable} are continuous-time dynamical systems of the form
\begin{equation}\label{eq:neuralODE_i}
    \dot{u}(t) := \dv{}{t}u(t) = f(u(t), t, \theta), \quad t\in[0,T],
\end{equation}
with a specific choice of vector field \( f \), that model the forward pass of a neural network as the evolution of a feature vector \( u(t) \). The vector field \( f \) is chosen to be a neural network parameterised by \( \theta \), for ease of computation, hence the name neural ODE. For a fixed input $u_0$, the output of the neural network defined through the neural ODE \eqref{eq:neuralODE_i} is given by the solution $u(T)$ of \eqref{eq:neuralODE_i} corresponding to the initial value $u_0$.

The map $u_0\to u(T)$ is called flow map of the neural ODE at time $T$. Therefore, the neural network defined by the neural ODE \eqref{eq:neuralODE_i} is its flow map at time $T$. The ability to find the exact flow map is highly dependent on the structure of the differential equation. For the neural ODE \eqref{eq:neuralODE_i}, the exact flow map is not known, thus numerical methods are needed for its approximation. The numerical integration of the flow map leads to deep discrete architectures. For instance, applying the Euler method to \eqref{eq:neuralODE_i} over a partition \( 0=t_0< t_1 < \dots < t_N =T \) of $[0,T]$, with step size \( h = T/N \), results in a residual neural network \cite{ResNet}
\[
    u_{k+1} = u_k + h f(u_k, t_k, \theta), \quad k=0,1,\dots,N-1,
\]
where \( u_k \approx u(t_k) \), and the number of discretisation points \( N \) determines the depth of the neural network.

The neural ODE formulation of neural networks thus generalises residual and recurrent neural networks, and has been extensively explored in modern machine learning, achieving notable success across diverse applications with various choices of \( f \), including second-order damped oscillators \cite{rusch2021coupled,zhang2025stuart}, state-space models \cite{gu2023mamba,gu2020hippo}, diffusion-based generative models \cite{karras2022elucidating,song2021score}, and graph neural networks \cite{chamberlain2021grand,nguyen2024coupled}.

One of the key advantages of formulating neural networks via ODEs is the ability to study their theoretical properties through dynamical systems analysis, particularly in terms of stability, contractivity, conservation laws, and approximation \cite{celledoni2021structure,de2025stability, guglielmi2024contractivity}.

In this paper, we focus on the approximation properties of neural ODEs, namely their density properties into a larger space of functions, as the space $\mathcal{C}(\R^m,\R^n)$ of continuous functions from $\R^m$ to $\R^n$, $m,n\in\N$ fixed. We notice that flow maps of differential equations have the same input and output dimensions. Since $m$ and $n$ are usually different, we thus need to embed an $m$-dimensional input into the latent space of the neural ODE, and to project the output of the neural ODE into the $n$-dimensional output space. These embedding and projection operations are usually carried out with simple affine maps.

If we define
\begin{itemize}
    \item $d$ the dimension of the latent space of the neural ODE \eqref{eq:neuralODE_i},
    \item $\ell_1(u) = A_1 u + b_1$ an affine map from $\R^m$ to $\R^d$, with $A_1\in\R^{d\times m}$, $b_1\in\R^d$,
    \item $\ell_2(u) = A_2 u + b_2$ an affine map from $\R^d$ to $\R^n$, with $A_2\in\R^{n\times d}$, $b_2\in\R^n$,
\end{itemize}
and we denote by $\phi:\R^d\to\R^d$ the flow map at time $T$ of the neural ODE \eqref{eq:neuralODE_i}, then we are studying the approximation properties of the space of functions $\varphi:\R^m\to\R^n$ defined as
\begin{equation}\label{eq:a}
      \varphi(x) = \ell_2 \circ \phi \circ \ell_1(x), \quad x\in\R^m.
\end{equation}
The function $\varphi$ looks like a shallow neural network, but it is actually a deep neural network because the computation of $\phi$ requires the evaluation of a deep neural network architecture. Nevertheless, by looking at $\varphi$ as a shallow neural network, we can employ the wide and well-developed approximation theory of shallow neural networks, together with dynamical system analysis, to study the approximation properties of neural ODEs.

In this article, we consider one-layer weight-tied neural ODEs of the form
\begin{equation*}\label{eq:neuralODE1}
    \dot{u}(t) = \sigma (A u(t) + b), \quad t\in[0,T],
\end{equation*}
where $\sigma:\R\to\R$ is a smooth activation function applied entry-wise, $T>0$ is the time horizon, $A\in\R^{d\times d}$ is the weight matrix and $b\in\R^d$ is the bias vector. A composition of functions $\sigma (A_i u + b_i)$, $i=1,2,\ldots$, allows us to obtain a multi-layer neural ODE, and the results in this article can be generalised to this setting as well.

\subsection{Our contribution}
We study the approximation properties of neural network architectures of the form \eqref{eq:a}. These networks can be thought of as shallow networks with arbitrary width and a nonstandard choice of activation function: the flow map of a neural ODE \cite{chen2018neural}. 

As in \cite{cybenko1989approximation,hornik1991approximation,hornik1989multilayer,leshno1993multilayer}, we consider the compact convergence topology, 

\begin{definition}\label{de:uap}
    Let $m,n\in\N$ and $\|\cdot\|$ a norm on $\mathbb{R}^n$. A set of functions $\mathcal{H}$ from $\R^m$ to $\R^n$ is a universal approximator for $\cC(\R^m,\R^n)$ under the compact convergence topology if, for any $f\in \cC(\R^m,\R^n)$, for any $K\subset\R^m$ compact, and for any $\varepsilon>0$, there exists $\varphi \in \mathcal{H}$ such that
    \begin{equation}\label{eq:app_err}
        \|f-\varphi\|_{\infty, K} := \max_{x\in K} \|f(x)-\varphi(x)\| \leq \varepsilon.
    \end{equation}
\end{definition}

In Definition \ref{de:uap}, we leave the norm $\|\cdot\|$ as an unspecified norm on $\mathbb{R}^n$. Later derivations will focus on the $\ell^2$ norm $\|\cdot\|_2$. By using this topology, we prove that the space of shallow neural networks with a flow map as activation function is a universal approximator for the space of continuous functions $\mathcal{C}(\R^m,\R^n)$. Furthermore, we empirically show that neural networks of this form have a better parameter efficiency\footnote{A neural network has better parameter efficiency than another one with the same number of parameters if it is more expressive than the latter.} than standard shallow networks.

We then show that constraining either the Lipschitz constant of the flow map or the norm of the weights does not affect the universal approximation property. %Then we introduce some constraints on the network parameters, specifically a constraint on the Lipschitz constant of the flow map and a unit norm constraint on the weight matrices of the outer layers. We show that these constraints alone do not affect the universal approximation property, whereas together they do. 
Constraining both leads to a loss in expressiveness, that we quantify by deriving an upper and a lower bound to \eqref{eq:app_err}. % that tell us respectively how accurately at most we can approximate a continuous function with a constrained shallow neural network and how much at most the accuracy deteriorates.

\subsection{Existing literature on the approximation theory of neural networks}
Neural networks are a common choice for approximating continuous functions in various applications, as their ability to approximate has been validated by numerous studies over the past few decades, and this remains an active area of research.

In 1989, Cybenko \cite{cybenko1989approximation} and Hornik et al. \cite{hornik1989multilayer} showed that shallow neural networks with arbitrary width and sigmoidal activation functions, i.e. increasing functions $\sigma:\R\to\R$ with $\lim_{x\to-\infty}\sigma(x)=0$ and $\lim_{x\to\infty}\sigma(x)=1$, are universal approximators for continuous functions. In 1991, Hornik \cite{hornik1991approximation} showed that it is not the specific choice of the activation function but rather the compositional nature of the network that gives shallow neural networks with arbitrary width the universal approximation property. In 1993, Leshno et al. \cite{leshno1993multilayer} and, in 1999, Pinkus \cite{pinkus1999approximation} showed that shallow neural networks with arbitrary width are universal approximators if and only if their activation function is not a polynomial. Other relevant references for the approximation properties of shallow neural networks of arbitrary width are \cite{barron1993universal,funahashi1989approximate,hassoun1995fundamentals,haykin1998neural}.

The recent breakthrough of deep learning has stimulated much research on the approximation ability of arbitrarily deep neural networks. Gripenberg in 2003 \cite{gripenberg2003approximation}, Yarotsky, Lu et al. and Hanin and Selke in 2017 \cite{ hanin2017approximating,lu2017expressive,yarotsky2017error} studied the universal approximation property of deep neural networks with arbitrary depth, bounded width, and with ReLU activation functions. Kidger and Lyons \cite{kidger2020universal}, in 2020, extended those results to arbitrarily deep neural networks with bounded width and arbitrary activation functions. Park \cite{park2021minimum}, in 2021, improved the results by Hanin \cite{hanin2017approximating} for deep ReLU neural networks with bounded width, and Cai \cite{cai2023achieve} in 2023 proved the approximation property for deep neural networks with bounded width and with the LeakyReLU activation function. Further research has explored the approximation capabilities of ReLU networks for a variety of function classes, including continuous functions \cite{shen2019deep,yarotsky2017error,yarotsky2018optimal}, smooth functions \cite{lu2021deep,yarotsky2020phase}, piecewise-smooth functions \cite{petersen2018optimal}, shift-invariant spaces \cite{yang2022approximation}, and band-limited functions \cite{montanelli2021deep}. Other valuable references for the approximation properties of arbitrarily deep neural networks of bounded width are \cite{johnson2019deep,juditsky2009nonparametric,kratsios2022universal,poggio2017and}.

In 1999, Maiorov and Pinkus \cite{maiorov1999lower} studied networks of bounded depth and width. They showed that there exists an analytic sigmoidal activation function such that neural networks with two hidden layers and bounded width are universal approximators. In 2018, Guliyev and Ismailov \cite{guliyev2018approximation} designed a family of two hidden layer feed-forward networks with fixed width and a specific sigmoidal activation function that is a universal approximator of continuous functions. In the same year, Guliyev and Ismailov \cite{guliyev2018on} also constructed shallow neural networks with bounded width that are universal approximators for univariate functions. However, this construction does not apply to multivariate functions. In 2022, Shen et al. \cite{shen2022optimal} computed precise quantitative information on the depth and width required to approximate a target function by deep and wide ReLU neural networks.

% \begin{figure}[ht]
%     \centering
%     \begin{subfigure}{0.48\textwidth}
%         \centering
%         \includegraphics[width=\textwidth]{shallow.pdf}
%         \caption{Shallow neural network with width 5.}
%     \end{subfigure}
%     \hfill
%     \begin{subfigure}{0.5\textwidth}
%         \centering
%         \includegraphics[width=\textwidth]{deep.pdf}
%         \caption{Deep neural network with depth 4 and width 10.}
%     \end{subfigure}
%     \caption{Graphical representation of neural networks.}
%     \label{fig:shallowvsdeep}
% \end{figure}

\subsection{Preliminaries}
In the following, we define the subspace $\mathcal{H}$ as the space of shallow neural networks of arbitrary width with an activation function given by the flow map of a neural ODE at time $t=1$.

Let $d\in\N$ and $\sigma:\R\to\R$ be an activation function that satisfies the following assumption.
\begin{assumption}\label{ass:1}
    $\sigma$ is differentiable almost everywhere and $\sigma'(\R)\subset[\alpha,1]$, with $0<\alpha\le1$.
\end{assumption}

\begin{remark}
    An example of such an activation function is the LeakyReLU
    \[
        \sigma(z) =
        \begin{cases}
            z, \qquad &\mbox{if } z \geq 0,\\
            \alpha z, \qquad &\mbox{otherwise},
        \end{cases}
    \]
    or a smoothed version of it, like
    \begin{equation}\label{eq:act_fun}
        \sigma(z) =
        \begin{cases}
            z, \qquad &\mbox{if } z \geq 0,\\
            \tanh{z}, \qquad &\mbox{if }  -\bar{z} \le z < 0,\\
            \alpha z + \beta, \qquad &\mbox{otherwise},
        \end{cases}
    \end{equation}
    where $\bar{z}>0$ is such that $\tanh'{(\pm\bar{z})}=\alpha=0.1$ and $\beta\in\R$ such that $\alpha (-\bar{z}) + \beta = \tanh{(-\bar{z})}$ (see Figure \ref{fig:custom_act_fun}).
    \begin{figure}[ht]
        \centering
        \includegraphics[width=0.5\textwidth]{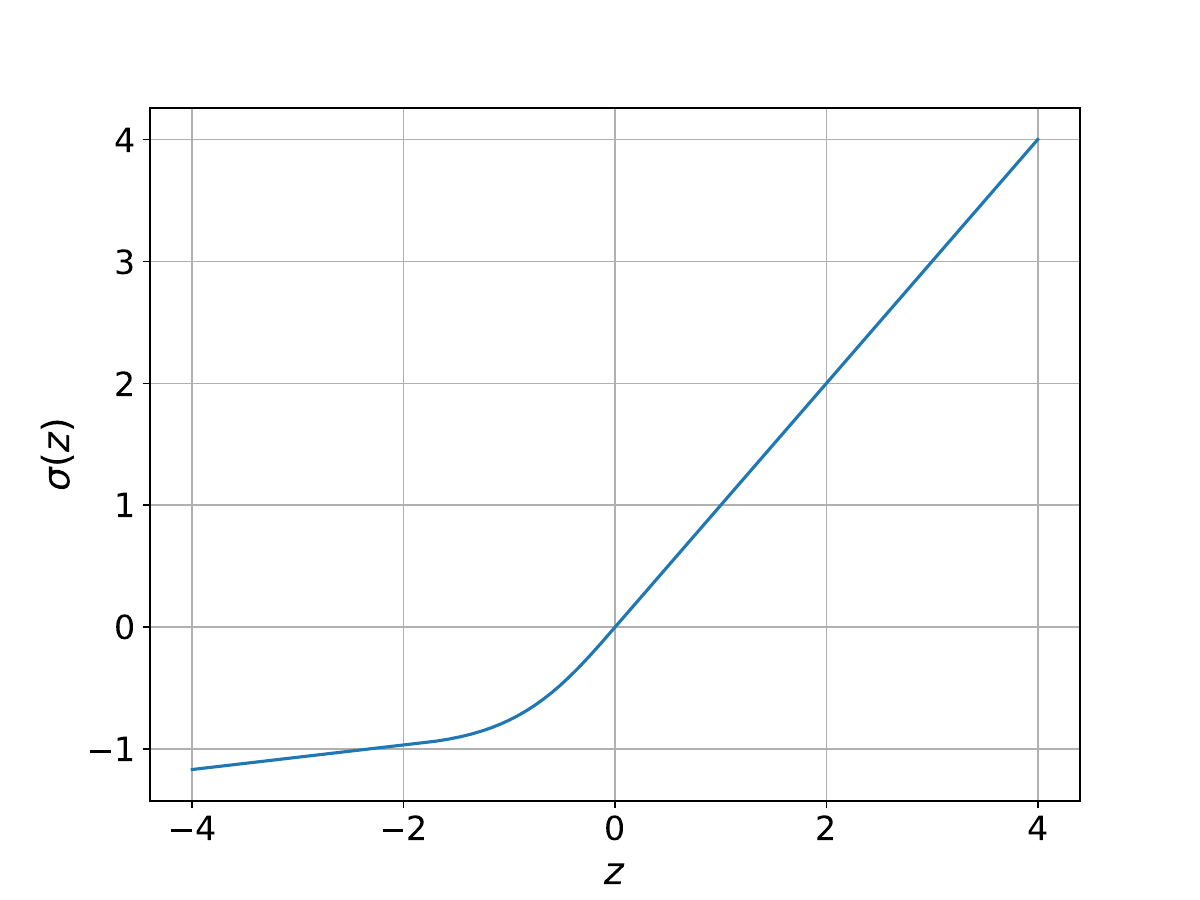}
        \caption{Smoothed Leaky Rectified Linear Unit (LeakyReLU) with minimal slope $\alpha=0.1$, as defined in \eqref{eq:act_fun}.}
        \label{fig:custom_act_fun}
    \end{figure}
\end{remark}

We consider the shallow neural network
\begin{equation}\label{eq:snn}
    \varphi(x) = A_2 \phi (A_1 x + b_1) + b_2, \quad x\in\R^m,
\end{equation}
with $A_1\in\R^{d\times m}$, $b_1\in\R^d$, $A_2\in\R^{n\times d}$, $b_2\in\R^n$, and the activation function $\phi:\R^d\to\R^d$ is the time$-1$ flow map of the neural ODE
\begin{equation}\label{eq:nODE}
    \dot{u}(t) = \sigma( A u(t) + b ), \quad t\in[0,1],
\end{equation}
with $A\in\R^{d\times d}$, $b\in\R^d$, $u(t)\in\R^d$ for all $t\in[0,1]$, and $\sigma$ is applied entrywise. If we define the affine functions
\[
    \ell_1(x) = A_1 x + b_1, \quad x\in\R^m, \qquad \text{and} \qquad \ell_2(x) = A_2 x + b_2, \quad x\in\R^d,
\]
we can rewrite $\varphi$ as follows
\[
    \varphi(x) = \ell_2 \circ \phi \circ \ell_1 (x), \quad x\in\R^m,
\]
i.e. as the composition of three functions: the affine layer $\ell_1$, the activation function $\phi$, and the affine layer $\ell_2$.

\begin{remark}
    From here on, when referring to the flow map of the neural ODE \eqref{eq:nODE}, we always consider it at time $1$.
\end{remark}

Before formally introducing the space $\mathcal{H}$, we need the following definitions.

\begin{definition}
    Fix $p, q \in \N$. We define the space of affine transformations from $\R^p$ to $\R^q$ as
    \[
        \Ac(\R^p,\R^q) = \{x\in\R^p \to A x + b \in \R^q \ : \ A\in\R^{q\times p},\ b\in\R^{q} \}.
    \]
\end{definition}

\begin{definition}
    Let $\phi_{A,b}$ be the flow map of the neural ODE \eqref{eq:nODE} for fixed $A\in\R^{d\times d}$ and $b\in\R^d$. We define the space of neural ODE flow maps as
    \[
        \mathcal{F}_d = \{ \phi_{A,b}:\R^d\to\R^d \ : \ A\in\R^{d\times d}, \ b\in\R^d \},
    \]
    where $A\in\R^{d\times d}$ is the weight matrix and $b\in\R^d$ is the bias vector.
\end{definition}

Then we can define the subspace $\mathcal{H}$ of $\mathcal{C}(\R^m,\R^n)$ as follows.

\begin{definition}
    We define the space of shallow neural networks \eqref{eq:snn} from $\R^m$ to $\R^n$ with arbitrary width and with activation function defined as a neural ODE flow map as
    \[
        \mathcal{H} = \{ \ell_2 \circ \phi \circ \ell_1 \ : \ d\in\N,\ \ell_1\in\Ac(\R^m,\R^d),\ \phi\in\mathcal{F}_d,\ \ell_2\in\Ac(\R^d,\R^n) \}.
    \]
\end{definition}

Eventually, we recall the definition of logarithmic norm.

\begin{definition}[see \cite{soderlind2006logarithmic,soderlind2024logarithmic}]\label{def:log_norm}
    Given $d\in\N$ and a matrix $A\in\R^{d\times d}$, denoted by $\|\cdot\|$ a matrix norm induced by a vector norm, we define the logarithmic norm of $A$ as
    \[
        \mu(A) = \lim_{h\to0^+} \frac{\|I+hA\|-1}{h}.
    \]
\end{definition}

\begin{remark}
    If the matrix norm is induced by the 2-norm $\|\cdot\|_2$, then the logarithmic 2-norm of $A$, denoted by $\mu_2(A)$, coincides with the largest eigenvalue of its symmetric part $\frac{1}{2}(A+A^\top)$.
\end{remark}

\subsection{Paper organisation}
The paper is organised as follows. In Section \ref{sec:uat}, we prove that the space of shallow neural networks with arbitrary width and the proposed choice of activation function is a universal approximator for the space of continuous functions under the compact convergence topology. We conclude the section by introducing two constraints that improve the stability of the proposed shallow neural network at the risk of losing the universal approximation property. The former is a constraint on the Lipschitz constant of the activation function to restrict its expansivity. The latter is a constraint of fixed 2-norm of the weight matrices $A_1$ and $A_2$ to ensure that the restricted expansivity of the activation function is not compensated by the increased expansivity in the linear layers. In Section \ref{sec:approx_bounds}, we derive approximation bounds that tell us how accurately we can approximate a continuous function with a shallow neural network satisfying both constraints. In Section \ref{sec:uat_unc}, we show that the universal approximation property holds if either the constraint on the Lipschitz constant of the activation function or the fixed norm constraint on the weight matrices is considered.

\section{Universal approximation theorem}\label{sec:uat}
We are ready to state the first main result of this paper.

\begin{theorem}\label{th:univ_approx_1}
    The space of functions
    \[
        \mathcal{H} = \{ \ell_2 \circ \phi \circ \ell_1 \ : \ d\in\N,\ \ell_1\in\Ac(\R^m,\R^d),\ \phi\in\mathcal{F}_d,\ \ell_2\in\Ac(\R^d,\R^n) \}
    \]
    is a universal approximator for $\cC(\R^{m},\R^{n})$ under the compact convergence topology.
\end{theorem}

To prove Theorem \ref{th:univ_approx_1}, we need the following definition and lemmas.

\begin{definition}
    Let $\mathcal{M}$ denote the set of real functions which are in $L^\infty_{loc}(\R)$ and have the following property: the closure of their sets of points of discontinuity is of zero Lebesgue measure.
\end{definition}
For each function $\sigma\in\mathcal{M}$, interval $[a,b]\subset\mathbb{R}$, and scalar $\delta >0$, there exists a finite number of open intervals, the union of which we denote by $U$, of measure $\delta$, such that $\sigma$ is uniformly continuous on $[a,b]\setminus U$.
%The property the functions in $\mathcal{M}$ satisfy implies that, if $\sigma\in\mathcal{M}$, $[a,b]\subset\R$ is an interval and $\delta>0$, there exists a finite number of open intervals, the union of which we denote by $U$, of measure $\delta$, such that $\sigma$ is uniformly continuous on $[a,b] \backslash U$. 
This is important in the proof of the next lemma.

\begin{lemma}[Theorem 1, page 10, \cite{leshno1993multilayer}]\label{lem:univ_approx}
    Let $\sigma\in\mathcal{M}$. Set
    \[
        \Sigma_{m,1} = \left\{ x\in\R^m \to \sum_{i=1}^k c_i \sigma (a_i \cdot x + b_i) + d \ : \ k\in\N,\ a_i\in\R^m,\ b_i, c_i, d \in \R,\ \forall i = 1,\dots,k \right\}.
    \]
    Then $\Sigma_{m,1}$ is dense in $\cC(\R^m,\R)$ under the compact convergence topology if and only if $\sigma$ is not a polynomial.
\end{lemma}

Lemma \ref{lem:univ_approx} is the first universal approximation theorem by Leshno et al. \cite{leshno1993multilayer} stating a very simple condition for the space $\Sigma_{m,1}$ of shallow neural networks to be a universal approximator for $\cC(\R^m,\R)$, i.e. the activation function $\sigma:\R\to\R$ has not to be a polynomial. All commonly used activation functions, such as ReLU, LeakyReLU, and $\tanh$, satisfy this requirement.

\begin{remark}\label{rem:ua}
    Lemma \ref{lem:univ_approx} implies the universal approximation property of the space of functions
    \[
        \Sigma_{m,n} = \left\{ x\in\R^m \to C \sigma (A x + b) + d \in\R^n \ : \ k\in\N,\ A\in\R^{kn\times m},\ b\in\R^{kn},\ C\in\R^{n\times kn},\  d\in\R^n \right\}
    \]
    in $\cC(\R^m,\R^n)$. Indeed, if
    \[
        F(x) = C \sigma (A x + b) + d, \quad x\in\R^m,
    \]
    with $\sigma\in\mathcal{M}$ applied entrywise,
    \[
        A = \begin{bmatrix}
                A_1 \\
                A_2 \\
                \vdots \\
                A_n
            \end{bmatrix},\
        b = \begin{bmatrix}
                b_1 \\
                b_2 \\
                \vdots \\
                b_n
            \end{bmatrix},\
        C = \begin{bmatrix}
                c_1^\top & 0        & \cdots & 0 \\
                0        & c_2^\top & \cdots & 0 \\
                \vdots   &          & \ddots & \vdots  \\
                0        & \cdots   & 0      & c_n^\top
            \end{bmatrix},
    \]
    $A_i\in\R^{k\times m}$, $b_i\in\R^k$ and $c_i\in\R^k$, $i=1,\dots,n$, and $d\in\R^n$, then the $i$-th component of $F$ is
    \[
        F_i(x) = c_i^\top \sigma (A_i x + b_i) + d_i = \sum_{j=1}^k c_{ij} \sigma (a_{ij}\cdot x + b_{ij}) + d_i.
    \]
    We notice that $F_i\in\Sigma_{m,1}$ for all $i=1,\dots,n$, and since $\Sigma_{m,1}$ is dense in $\cC(\R^m,\R)$ under the compact convergence topology if $\sigma$ is not a polynomial, then $\Sigma_{m,n}$ is dense in $\cC(\R^m,\R^n)$ under the same topology.

    The important hypotheses to remember here are that the activation function $\sigma:\R\to\R$ is applied entrywise and that it is not a polynomial.
\end{remark}

The next Lemma follows. It states that the only scalar ordinary differential equations whose flow map is a polynomial in $x$ are the linear ordinary differential equations, and their flow map is a linear polynomial in $x$.

\begin{lemma}[Proposition 4.2, page 179 in \cite{bass1985polynomial}]\label{lem:pol_flow}
    Let $\phi:\R\times\R\to\R$, $\phi=\phi(t,x)$, be a real flow map polynomial in $x$ for all times $t$, and $\dot{u}(t)=f(u(t))$ the corresponding differential equation, i.e. $\phi(t,u(0))=u(t)$. Then there are constants $a,b\in\R$ such that $f(u)=au+b$ and
    \[
        \phi(t,x)=e^{at}x+\frac{b}{a}(e^{at}-1),
    \]
    where we interpret $(e^{at}-1)/a$ to be $t$ for $a=0$.
\end{lemma}

Lemma \ref{lem:pol_flow} is crucial for the proof of Theorem \ref{th:univ_approx_1}.

\begin{proof}[Proof of Theorem \ref{th:univ_approx_1}]
    We consider the time$-1$ flow map $\phi_\lambda:\R^d\to\R^d$, $\lambda\in\R$, of the differential equation
    \begin{equation}\label{eq:snODE}
        \dot{u}(t) = \sigma( u(t) + \lambda e), \quad u(t)\in\R^d,\ t\in[0,1],
    \end{equation}
    with $\sigma:\R\to\R$ satisfying Assumption \ref{ass:1} and $e\in\R^d$ the vector of all ones. Assumption \ref{ass:1} guarantees that $\sigma$ is Lipschitz continuous and thus the existence and uniqueness of the solution of \eqref{eq:snODE}. This is a particular case of neural ODE \eqref{eq:nODE} with $A\in\R^{d\times d}$ the identity matrix and $b\in\R^d$ a vector with coinciding entries. Written componentwise, \eqref{eq:snODE} turns into the $d$ scalar uncoupled ODEs
    \[
        \dot{u}_i(t) = \sigma( u_i(t) + \lambda), \quad u_i(t)\in\R,\ i=1,\dots,d,\ t\in[0,1].
    \]
    %which shows that each component of the equation acts independently on each component of the solution. 
    Therefore all the components of the flow map $\phi_\lambda:\R^d\to\R^d$ coincide, i.e. $\phi_\lambda = ((\phi_\lambda)_i)_{1\le i \le d}$, with $(\phi_\lambda)_i=(\phi_\lambda)_j$ for all $i\ne j$.

    We define the space
    \[
        \mathcal{F}_{d,0} = \{ \phi_\lambda:\R^d\to\R^d \ : \ \lambda\in\R \}\subset\mathcal{F}_d
    \]
    of flow maps $\phi_\lambda:\R^d\to\R^d$ of the differential equations \eqref{eq:snODE}, and the space
    \begin{equation*}
        \mathcal{H}_0 = \{ \ell_2 \circ \phi \circ \ell_1 \ : \ d\in\N,\ \ell_1\in\Ac(\R^m,\R^d),\ \phi\in\mathcal{F}_{d,0},\ \ell_2\in\Ac(\R^d,\R^n) \}\subset\mathcal{H}.
    \end{equation*}
    A generic function $\varphi$ of $\mathcal{H}_0$ has the form
    \[
        \varphi(x) = \ell_2 \circ \phi_\lambda \circ \ell_1(x) = A_2 \phi_\lambda ( A_1 x + b_1 ) + b_2, \quad x\in\R^m,
    \]
    with $\ell_1(u) = A_1 u + b_1$, $\lambda\in\R$, and $\ell_2(u) = A_2 u + b_2$. Since $\phi_\lambda$ has components all equal, it can be interpreted as a scalar univariate activation function applied entrywise, that is the time--1 flow map of the scalar differential equation
    \[
        \dot{u}(t) = \sigma( u(t) + \lambda), \quad u(t)\in\R,\ t\in[0,1].
    \]
    Therefore, since $\phi_\lambda:\R\to\R$ is not a polynomial for Lemma \ref{lem:pol_flow}, Lemma \ref{lem:univ_approx} holds and $\mathcal{H}_0$ is a universal approximator for $\cC(\R^m,\R^n)$ under the compact convergence topology.

    Finally, since $\mathcal{H}_0\subset\mathcal{H}$, it follows for transitivity that $\mathcal{H}$ is a universal approximator for $\cC(\R^m,\R^n)$ under the compact convergence topology.
\end{proof}

\begin{remark}
    Notice that in the proof of Theorem \ref{th:univ_approx_1}, Assumption \ref{ass:1} can be relaxed, and one may simply require that $\sigma$ is Lipschitz continuous.
\end{remark}

\subsection{Numerical experiments on the approximation efficiency}
Theorem \ref{th:univ_approx_1} shows the universal approximation property of shallow neural networks whose activation function is defined as the flow map of neural ODE \eqref{eq:nODE}. To the best of our knowledge, such activation functions have not been considered so far.

The proof of Theorem \ref{th:univ_approx_1} does not leverage that $\phi\in\mathcal{F}_d$ is a flow map, and treats it as a scalar activation function applied entrywise. Despite this, it is evident that a flow map is able to express much more complicated maps than a scalar activation which does not mix the components of the input vector, and this is confirmed by a set of numerical experiments we now describe. We fix $\sigma=\mathrm{LeakyReLU}$ and consider the three following architectures:
\begin{enumerate}
    \item a network with a single hidden layer, i.e. $\ell_2\circ \sigma \circ \ell_1$, with $\ell_1\in\mathcal{A}(\mathbb{R}^m,\mathbb{R}^d),\,\ell_2\in\mathcal{A}(\mathbb{R}^d,\mathbb{R}^n)$;
    \item a neural ODE $\ell_2\circ\phi\circ\ell_1$, where $\phi$ is the time--$1$ flow map of $\dot{u}=\sigma(Au+b)$, $A\in\mathbb{R}^{d\times d},\,b\in\mathbb{R}^d$;
    \item a network with two hidden layers and the same number of parameters as the neural ODE, i.e. a map of the form $\ell_3\circ \sigma \circ \ell_2 \circ \sigma \circ \ell_1$, where $\ell_1\in\mathcal{A}(\mathbb{R}^m,\mathbb{R}^d),\,\ell_2\in\mathcal{A}(\mathbb{R}^d,\mathbb{R}^d),\,\ell_3\in\mathcal{A}(\mathbb{R}^d,\mathbb{R}^n)$.
\end{enumerate}
The neural ODE is discretised in time with the explicit Euler method and step size $1/20$. We compare these models on the task of approximating the scalar function $f(x)=\sin(10x)+x$, for $x\in [0,1]$. To test the efficiency of these models, we consider the training points $N$ in the set $\{10, 50, 100, 500, 1000\}$, and the number of hidden neurons $d$ in $\{5, 10, 50, 100\}$. We train all the models for $1000$ epochs with the Adam optimiser, the mean-squared error loss function, and a cyclic learning rate scheduling based on cosine annealing. We test all the models on a fixed set of $1000$ test points to have unbiased experiments. Furthermore, once a pair $(N,d)$ is fixed, we generate a single training set for all three models.
\begin{figure}
    \centering
    \includegraphics[width=.8\textwidth]{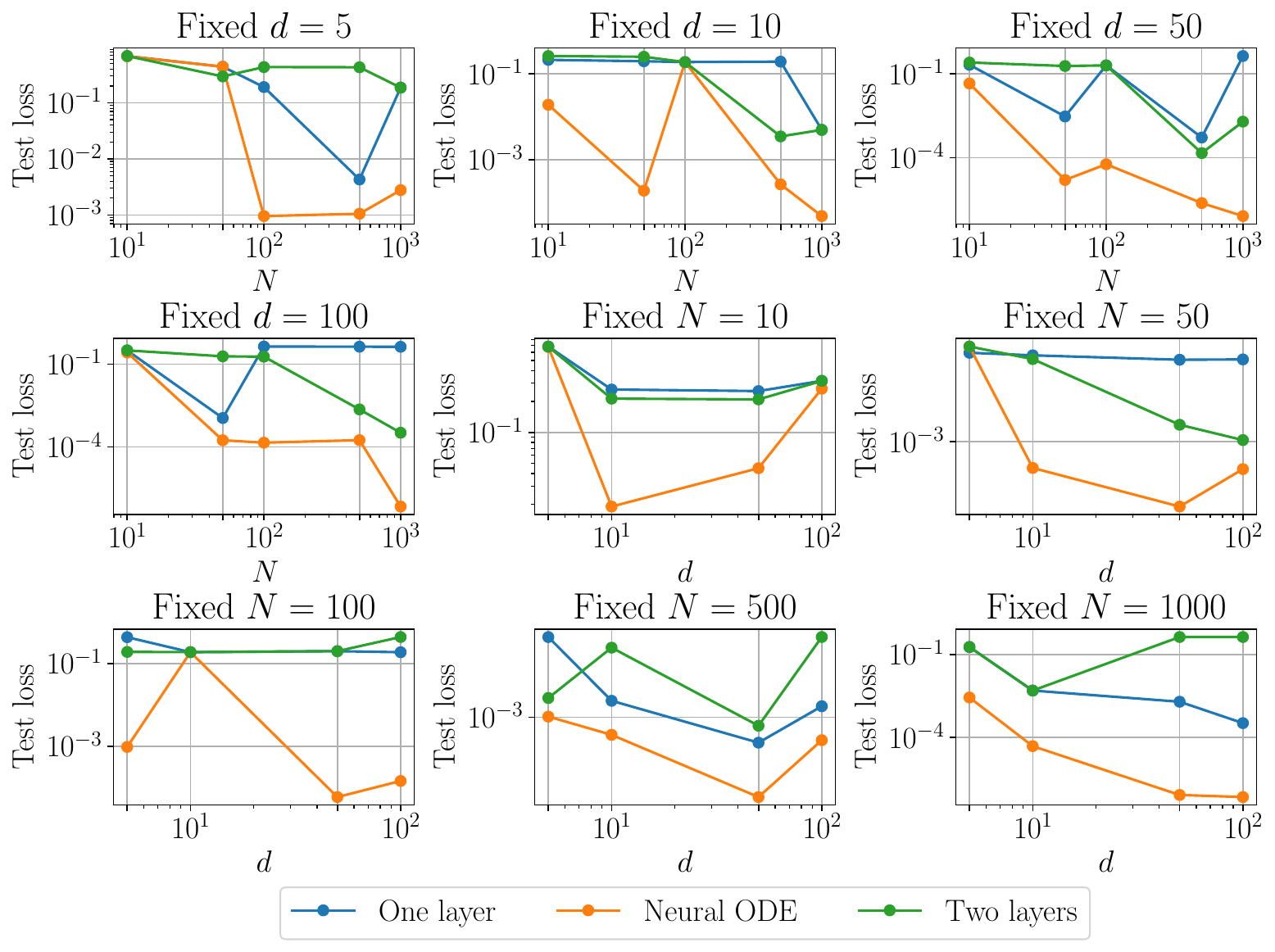}

    \caption{Comparison of the test mean squared error with varying numbers of training samples $(N)$ and hidden neurons $(d)$.}
    \label{fig:efficiencyNODE}
\end{figure}
The experiments in Figure \ref{fig:efficiencyNODE} show that the neural ODE tends to be more efficient than the other two models with respect to the number of training samples needed to train to a good accuracy. It is also apparent that the neural ODE is better at leveraging the additional parameters introduced by increasing the number of hidden neurons than the third model. This experimental analysis suggests that one could leverage the expressive power of flow maps to get better approximation rates for neural ODEs than with shallow neural networks of the considered form.

\subsection{Remarks on stability}
Classical activation functions are 1--Lipschitz, while flow maps of neural ODEs \eqref{eq:nODE} are Lipschitz functions but their Lipschitz constant depends on the weight matrix $A\in\R^{d\times d}$. In particular, in \cite{celledoni2021structure,de2025stability,guglielmi2024contractivity} it has been shown that the Lipschitz constant $L$ of the flow map $\phi:\R^d\to\R^d$ of the neural ODE \eqref{eq:nODE} is bounded by
\begin{equation}\label{eq:bound1}
    L \le \exp {\max_{D\in\Omega_\alpha} \mu_2(DA)},
\end{equation}
where $\alpha$ is defined in Assumption \ref{ass:1} and 
\begin{equation}\label{eq:omega}
    \Omega_\alpha = \{ D\in\R^{d\times d} \ : \ D \text{ is diagonal and } \alpha \le D_{ii} \le 1, \ \forall i = 1, \dots, d \}.
\end{equation}
The bound on the Lipschitz constant \eqref{eq:bound1} allows us to state the following stability bound for the shallow neural network \eqref{eq:snn}: if $x_1,x_2\in\R^m$ are two inputs, then
\begin{align}\label{eq:bound2}
    \begin{split}
    \|\varphi(x_1)-\varphi(x_2)\|_2 &\le \|A_2\|_2 L \|A_1\|_2 \|x_1-x_2\|_2 \\
                                    &\le \|A_1\|_2 \|A_2\|_2 \exp {\max_{D\in\Omega_\alpha} \mu_2(DA)} \|x_1-x_2\|_2,
    \end{split}
\end{align}
i.e. a perturbation in input is amplified at most by a factor $\|A_1\|_2 \|A_2\|_2 \exp {\max_{D\in\Omega_\alpha} \mu_2(DA)}$ in output. This information is very useful in light of the fragility of neural networks to adversarial attacks \cite{biggio2013evasion,goodfellow2014explaining,szegedy2013intriguing}, perturbations in input designed to mislead the neural network into making incorrect predictions, and it suggests an approach to improve the stability of the model \eqref{eq:snn}.

If we set $\|A_1\|_2\ge1$ a fixed constant and $\|A_2\|_2=1$, then the only term that controls the stability of the shallow neural network \eqref{eq:snn} is $\|A_1\|_2\exp {\max_{D\in\Omega_\alpha} \mu_2(DA)}$. In \cite{guglielmi2024contractivity}, an efficient numerical method has been proposed to compute $\max_{D\in\Omega_\alpha} \mu_2(DA)$, and therefore $\exp {\max_{D\in\Omega_\alpha} \mu_2(DA)}$. We denote
\[
    \delta_\star = \max_{D\in\Omega_\alpha} \mu_2(DA).
\]
For shallow neural networks as in \eqref{eq:snn} trained for classification tasks, we notice in numerical experiments that unconstrained neural ODEs have $\delta_\star>0$, and thus they are intrinsically expansive and potentially unstable. For example, we train a shallow neural network $\varphi$ of the kind \eqref{eq:snn}, with $\|A_1\|_2\ge1$ a fixed constant and $\|A_2\|_2=1$, to correctly classify the images in the MNIST dataset. The model $\varphi$ has constant $\delta_\star\approx4.65$. See \cite[Section 5]{de2025stability} for more details on this numerical experiment. Having a constant $\delta_\star\approx4.65$ implies that the Lipschitz constant of the shallow neural network is at most $e^{\delta_\star}\|A_1\|_2\approx104.6\|A_1\|_2$, which is very high and allows some perturbations in input to be greatly amplified in output. Table \ref{tab:MNIST} shows the accuracy\footnote{The accuracy is the percentage of correctly classified images in the test set.} of the shallow neural network $\varphi$ as a function of the magnitude $\eta>0$ of the perturbation introduced in input. The considered perturbation is the FGSM (Fast Gradient Sign Method) adversarial attack \cite{biggio2013evasion,goodfellow2014explaining,szegedy2013intriguing}.
\begin{table}[ht]
    \centering
    \begin{tabular}{c c c c c c c c}
        \toprule
        $\eta$ & 0 & 0.02 & 0.04 & 0.06 & 0.08 & 0.10 & 0.12 \\
        \midrule
        $\varphi$ & 0.9767 & 0.9205 & 0.7470 & 0.4954 & 0.2872 & 0.1426 & 0.0620 \\
        \bottomrule
    \end{tabular}
    \caption{Accuracy of the shallow neural network $\varphi$ as a function of the magnitude $\eta$ of the FGSM adversarial attack.}
    \label{tab:MNIST}
\end{table}
We notice that the accuracy decays quickly as the magnitude $\eta$ of the perturbation increases.

Therefore, in \cite{de2025stability}, a numerical approach has been proposed to stabilise the shallow neural network $\eqref{eq:snn}$. It is sketched as follows. For a fixed $\delta<\delta_\star$, the numerical method computes one of the perturbations $\Delta\in\R^{d\times d}$ with smallest Frobenius norm such that
\[
    \max_{D\in\Omega_\alpha} \mu_2(D(A+\Delta)) = \delta.
\]
We write \lq\lq one of the perturbations\rq\rq\ because the underlying optimisation problem is not convex, see \cite[Remark 5.1]{guglielmi2024contractivity}. We have the following scenarios. The shallow neural network \eqref{eq:snn} is
\begin{itemize}
    \item slightly expansive if $\|A_1\|_2e^{\delta_\star}>1$ but not large, i.e. $\delta_\star>\log\frac{1}{\|A_1\|_2}$ but not large;
    \item nonexpansive if $\|A_1\|_2e^{\delta_\star}=1$, i.e. $\delta_\star=\log\frac{1}{\|A_1\|_2}$;
    \item contractive if $\|A_1\|_2e^{\delta_\star}<1$, i.e. $\delta_\star<\log\frac{1}{\|A_1\|_2}$.
\end{itemize}

Therefore, choosing a $\delta<\delta_\star$ and applying the numerical method proposed in \cite{de2025stability} to get $\max_{D\in\Omega_\alpha} \mu_2(D(A+\Delta)) = \delta$ allows us to get a shallow neural network of the kind \eqref{eq:snn} that is at most moderately expansive, i.e. stable. In particular, if
\begin{equation*}
    \dot{u}(t) = \sigma( (A+\Delta) u(t) + b ), \quad t\in[0,1],
\end{equation*}
is the stabilised neural ODE and $\bar{\phi}$ is its time--1 flow map, then the stabilised shallow neural network is
\begin{equation}\label{eq:ssnn}
    \bar{\varphi}(x) = A_2 \bar{\phi} (A_1 x + b_1) + b_2, \quad x\in\R^m,
\end{equation}
with the same notations of \eqref{eq:snn} and \eqref{eq:nODE}.

Nevertheless, enforcing stability or contractivity can degrade model accuracy. In classification tasks, in particular, requiring contractivity may significantly reduce the fraction of correctly classified test samples. This trade-off is inevitable: overly accurate models tend to be unstable \cite{gottschling2020troublesome}, while strongly stable models may lack expressiveness. Therefore it is important to quantify the loss of expressivity resulting from the inclusion of the perturbation $\Delta$ in the model \eqref{eq:snn} to control the Lipschitz constant of the flow map $\phi$ and from the constraint $\|A_1\|_2\ge1$ a fixed constant and $\|A_2\|_2=1$. This motivates the next sections.

\section{Approximation bounds}\label{sec:approx_bounds}
In this section, we investigate the approximation properties of neural networks with a constrained Lipschitz constant. We focus on the flow map $\phi$. However, to ensure its restricted expansivity is not compensated by the increased expansivity in the linear layers, we limit the norm of the matrix $A_1$ to a fixed constant larger than or equal to 1 and of the matrix $A_2$ to one. We introduce the following space.
\begin{definition}
    Fix $p, q \in \N$ and $c\ge1$. We define the space of fixed norm affine transformations from $\R^p$ to $\R^q$ as
    \[
        \Ac_c(\R^p,\R^q) = \{x\in\R^p \to A x + b \in \R^q \ : \ A\in\R^{q\times p},\ \|A\|_2=c,\ b\in\R^{q} \}.
    \]
\end{definition}
%First, in light of Theorem \ref{th:univ_approx_2} and Remark \ref{rem:comment}, we can state that
We then focus on the space of neural networks defined by
\[
    \mathcal{H}_{c,1} = \{ \ell_2 \circ \phi \circ \ell_1 \ : \ d\in\N,\ \ell_1\in\Ac_c(\R^m,\R^d),\ \phi\in\mathcal{F}_d,\ \ell_2\in\Ac_1(\R^d,\R^n) \}.
\]
%is a universal approximator for $\cC(\R^{m},\R^{n})$ under the compact convergence topology. 
We notice that $\mathcal{H}_{c,1}$ corresponds to $\mathcal{H}$ with $\ell_1\in\Ac_c(\R^m,\R^d)$ and $\ell_2\in\Ac_1(\R^d,\R^n)$ instead of $\ell_1\in\Ac(\R^m,\R^d)$ and $\ell_2\in\Ac(\R^d,\R^n)$.%; equivalently, $\mathcal{H}_1^\star$ is defined as $\mathcal{H}_{c,1}$ where $\psi\in\mathcal{F}_{d,1}$ is replaced by $\phi\in\mathcal{F}_d$.

We recall that the Lipschitz constant $L_\phi$ of the flow map $\phi$ of neural ODE \eqref{eq:nODE} $\dot{u}=\sigma(Au+b)$ satisfies the bound
\[
    L_\phi \le e^{\delta_\star} \quad \text{with} \quad \delta_\star = \max_{D\in\Omega_\alpha} \mu_2(DA).
\]
Furthermore, if needed, fixed $\delta < \delta_\star$, it is possible to compute one of the perturbation matrices $\Delta$ of minimal Frobenius norm such that $\max_{D\in\Omega_\alpha} \mu_2(D(A+\Delta)) = \delta$ (see \cite{de2025stability}, \cite[Remark 5.1]{guglielmi2024contractivity}), i.e. the Lipschitz constant $L_{\bar{\phi}}$ of the flow map $\bar{\phi}$ of neural ODE $\dot{u}=\sigma(\bar{A}u+b)$, with $\bar{A}=A+\Delta$, is such that
\[
    L_{\bar{\phi}} \le e^{\delta} \quad \text{with} \quad \delta = \max_{D\in\Omega_\alpha} \mu_2(D(A+\Delta)).
\]
Without constraints of any kind, the Lipschitz constant $L_\varphi$ of the shallow neural network \eqref{eq:snn} $\varphi(x) = A_2 \phi (A_1 x + b_1) + b_2$ satisfies the bound
\[
    L_\varphi \le \|A_2\|_2 L_\phi \|A_1\|_2 \le \|A_1\|_2 \|A_2\|_2 e^{\delta_\star}.
\]
If we impose the Lipschitz constant of the flow map of the neural ODE \eqref{eq:nODE} to be smaller than or equal to $e^{\delta}$ for a certain $\delta\in\R$, and we restrict the norm of the matrix $A_1$ to a fixed constant larger than or equal to 1 and of the matrix $A_2$ to one, then we lose the universal approximation property. For example, if we aim to approximate a function $f\in \cC(\R^m,\R^n)$ that is $10-$Lipschitz, say $f(x)=10x$, with a shallow neural network \eqref{eq:snn} with Lipschitz constant bounded by 5, we fail to approximate it arbitrarily well.

Nevertheless, if we define the approximation error on a compact subset $K$ of $\R^m$ between a continuous function $f\in \cC(\R^m,\R^n)$ and a shallow neural network $\psi$, see \eqref{eq:snn}, as
\begin{equation} \label{eq:approx_err}
    \|f-\psi\|_{\infty,K} := \max_{x\in K} \|f(x)-\psi(x)\|_2,
\end{equation}
then we can quantify how much the approximation accuracy degrades because of the perturbation matrix $\Delta$ by computing an upper bound to \eqref{eq:approx_err}, and how accurately at most a shallow neural network approximates a continuous function by computing a lower bound to \eqref{eq:approx_err}.

We need the following preliminary results. Their proofs are well-known but we write them for completeness.

\begin{lemma}\label{lem:gronwall}
    Given $a,b>0$, we consider the linear scalar differential inequality
    \[
        u'(t) \le a u(t) + b, \quad t\in[0,1],
    \]
    with $u:[0,1]\to\R$ and $u(0)=u_0$. Then
    \[
        u(t) \le e^{at}u_0 + \frac{b}{a}(e^{at}-1), \quad t\in[0,1].
    \]
\end{lemma}

\begin{proof}
    We start by defining the function
    \[
        z(s) := u(s)e^{-as}, \quad s\in[0,1].
    \]
    Differentiating $z(s)$ with respect to $s$, we get
    \[
        z'(s) = u'(s)e^{-as} - ae^{-as}u(s) \le (au(s)+b)e^{-as} - ae^{-as}u(s) = be^{-as}.
    \]
    Integrating both sides from 0 to $t\in[s,1]$, we obtain
    \[
        z(t) - z(0) \le \int_0^t b e^{-as} \text{d}s.
    \]
    Substituting the definition of $z(t)$ back into this inequality and solving the integral, we have
    \[
        u(t)e^{-at} - u(0) \le -\frac{b}{a} (e^{-at}-1).
    \]
    Eventually, multiplying both sides by $e^{at}$, we get
    \[
        u(t) \le e^{at}u_0 + \frac{b}{a}(e^{at}-1), \quad t\in[0,1],
    \]
    which concludes the proof.
\end{proof}

\begin{remark}
    An analogous version of Lemma \ref{lem:gronwall} holds with $\ge$ instead of $\le$.
\end{remark}

\begin{lemma}\label{lem:svd}
    If $A\in\R^{m\times n}$, with $m,n\in\N$, and $x\in\R^n$, then
    \[
        \sigma_{\min}(A) \|x\|_2 \le \|Ax\|_2 \le \sigma_{\max}(A)\|x\|_2,
    \]
    where, from here on, $\sigma_{\max}(\cdot)$ and $\sigma_{\min}(\cdot)$ denote the largest and the smallest singular values of a matrix respectively.
\end{lemma}

\begin{proof}
    By noticing that $\sigma_{\max}(A)=\|A\|_2$, the upper bound is obvious, so we focus on the lower bound. Let $A = U \Sigma V^\top$ be the singular value decomposition of the matrix $A$, with $U\in\R^{m\times m}$ a unitary matrix, $\Sigma\in\R^{m\times n}$ a diagonal matrix whose elements are the singular values $\sigma_1>\sigma_2>\dots>\sigma_{\min(m,n)}$ of $A$, and $V\in\R^{n\times n}$ is a unitary matrix. If $V=[v_1|v_2|\dots|v_n]$, with $v_i\in\R^n$ for $i=1,\dots,n$, then
    \begin{align*}
        \|Ax\|_2 &= \|U \Sigma V^\top x\|_2 = \|\Sigma V^\top x\|_2 = \sqrt{\sum_{i=1}^{\min(m,n)} \sigma_i^2 (v_i^\top x)^2} \\
        &\ge \sigma_{\min}(A) \sqrt{\sum_{i=1}^{\min(m,n)} (v_i^\top x)^2} = \sigma_{\min}(A) \|V^\top x\|_2 = \sigma_{\min}(A) \|x\|_2,
    \end{align*}
    which concludes the proof.
\end{proof}

\begin{lemma}\label{lem:l2n}
    If $A\in\R^{n\times n}$, with $n\in\N$, and $x\in\R^n$, then
    \[
        -\mu_2(-A) \|x\|_2 \le \langle x , A x \rangle_2 \le \mu_2(A) \|x\|_2.
    \]
\end{lemma}

\begin{proof}
    By definition of inner product $\langle \cdot , \cdot \rangle_2$,
    \[
        \langle x , A x \rangle_2 = x^\top A x = \frac{1}{2} ( x^\top A^\top x + x^\top A x ) = x^\top \frac{A+A^\top}{2} x =: x^\top \text{sym}(A) x = \langle x , \text{sym}(A) x \rangle_2.
    \]
    If $\lambda_1,\lambda_2,\dots,\lambda_n$ are the eigenvalues of $\text{sym}(A)$ with $u_1,u_2,\dots,u_n\in\R^n$ corresponding unit eigenvectors, then we can write
    \[
        x = \sum_{i=1}^n \alpha_i u_i, \quad \alpha_i\in\R, \qquad \text{and} \qquad \|x\|_2^2 = \sum_{i=1}^n \alpha_i^2.
    \]
    If $\delta_{ij}$ is the Kronecker delta, then
    \begin{align*}
        \langle x , A x \rangle_2 &= \langle \sum_{i=1}^n \alpha_i u_i , \text{sym}(A) \sum_{j=1}^n \alpha_j u_j \rangle_2 = \sum_{i,j=1}^n \alpha_i \alpha_j \langle u_i , \text{sym}(A) u_j \rangle_2 = \sum_{i,j=1}^n \alpha_i \alpha_j \lambda_j \langle u_i , u_j \rangle_2 \\
        &= \sum_{i,j=1}^n \delta_{ij} \alpha_i \alpha_j \lambda_j = \sum_{i=1}^n \lambda_i \alpha_i^2.
    \end{align*}
    We define
    \[
        \lambda_{\max}(\text{sym}(A)) = \max_{i=1,\dots,n} \lambda_i \quad \text{and} \quad \lambda_{\min}(\text{sym}(A)) = \min_{i=1,\dots,n} \lambda_i.
    \]
    Therefore, by definition \ref{def:log_norm},
    \[
        \langle x , A x \rangle_2 \le \lambda_{\max}(\text{sym}(A)) \sum_{i=1}^n \alpha_i^2 =  \mu_2(A) \|x\|_2^2,
    \]
    and
    \[
        \langle x , A x \rangle_2 \ge \lambda_{\min}(\text{sym}(A)) \sum_{i=1}^n \alpha_i^2 =  - \mu_2(-A) \|x\|_2^2,
    \]
    which concludes the proof.
\end{proof}

Furthermore, we consider the time$-t$ flow maps $\phi_t$ and $\bar{\phi}_t$, $t\in [0,1]$, where $\phi_1=\phi$ and $\bar{\phi}_1=\bar{\phi}$. To simplify the notation, we define $z(t):=\phi_t(A_1x+b_1)$ and $\bar{z}(t):=\bar{\phi}_t(A_1x+b_1)$.
\begin{itemize}
    \item $z:[0,1]\to\R^d$ is the solution of the initial value problem
        \begin{equation}\label{eq:orig_ode}
            \begin{cases}
                \dot{z}(t) = \sigma (A z(t) + b), \quad t\in[0,1], \\
                z(0) = A_1x+b_1,
            \end{cases}
        \end{equation}
    \item $\bar{z}:[0,1]\to\R^d$ is the solution of the initial value problem
        \begin{equation}\label{eq:pert_ode}
            \begin{cases}
                \dot{\bar{z}}(t) = \sigma (\bar{A} \bar{z}(t) + b), \quad t\in[0,1], \\
                \bar{z}(0) = A_1x+b_1.
            \end{cases}
        \end{equation}
\end{itemize}
We are ready to state the main results of this section.

\subsection{Approximation upper bound}
Some preliminary definitions and notations follow. For a shallow neural network $\varphi\in\mathcal{H}_{c,1}$, we define $\delta_\star := \max_{D\in\Omega_\alpha} \mu_2(DA)$ and, for a fixed $\delta<\delta_\star$, we recall that there exists a perturbation matrix $\Delta$ such that
\[
    \max_{D\in\Omega_\alpha} \mu_2(D(A+\Delta)) = \delta.
\]
We then define the perturbed matrix $\bar{A}$ and the stabilised shallow neural network $\bar{\varphi}$:
\begin{enumerate}
    \item $\bar{A} = A + \Delta$,
    \item $\bar{\varphi}(x) = A_2\bar{\phi}(A_1x+b_1)+b_2, \quad x\in\R^m$.
\end{enumerate}

\begin{theorem} \label{th:upp_bound}
    Given a function $f\in \cC(\R^m,\R^n)$, a compact subset $K$ of $\R^m$, and a shallow neural network $\varphi\in\mathcal{H}_{c,1}$ such that
    \[
        \|f-\varphi\|_{\infty,K} \le \varepsilon
    \]
    for some $\varepsilon>0$, there exist $C>0$ such that
    \[
        \|f-\bar{\varphi}\|_{\infty,K} \le C \sigma_{\max}(\Delta) \frac{e^{\delta} - 1}{\delta} + \varepsilon.
    \]
\end{theorem}

\begin{proof}
    We recall that
    \[
        \|f-\bar{\varphi}\|_{\infty,K} = \max_{x\in K} \|f(x)-\bar{\varphi}(x)\|_2.
    \]
    For a fixed $x\in K$, by the triangle inequality, we get
    \[
        \|f(x)-\bar{\varphi}(x)\|_2 \le \|f(x)-\varphi(x)\|_2 + \|\varphi(x)-\bar{\varphi}(x)\|_2 \le \|\varphi(x)-\bar{\varphi}(x)\|_2 + \varepsilon,
    \]
    so we need to estimate $\|\varphi(x)-\bar{\varphi}(x)\|_2$. By definition of $\varphi$ and $\bar{\varphi}$, we have that
    \begin{align}\label{eq:main_ineq}
        \begin{split}
        \|\varphi(x)-\bar{\varphi}(x)\|_2 &= \|A_2\phi(A_1x+b_1)+b_2-A_2\bar{\phi}(A_1x+b_1)-b_2\|_2 \\
        &\le \|A_2\|_2 \| \phi(A_1x+b_1)-\bar{\phi}(A_1x+b_1) \|_2 \\
        &=\| \phi(A_1x+b_1)-\bar{\phi}(A_1x+b_1) \|_2, 
        \end{split}
    \end{align}
    since $\|A_2\|_2=1$. Therefore, we need to estimate the difference of the flow maps of two differential equations from the same initial datum.
    
    We define $z(t)=\phi_t(A_1x+b_1)$ and $\bar{z}(t)=\bar{\phi}_t(A_1x+b_1)$ for all $t\in[0,1]$ as in \eqref{eq:orig_ode} and \eqref{eq:pert_ode}, with $\phi_1=\phi$ and $\bar{\phi}_1 = \bar{\phi}$. Then, for $t>0$, we compute
    \begin{align}\label{eq:imp_comp}
        \begin{split}
            \dv{}{t^+} \|z(t)-\bar{z}(t)\|_2 &= \lim_{h\to0^+} \frac{\|z(t+h)-\bar{z}(t+h)\|_2-\|z(t)-\bar{z}(t)\|_2}{h} \\
            &= \lim_{h\to0^+} \frac{\|z(t)+h\dot{z}(t)-\bar{z}(t)-h\dot{\bar{z}}(t)\|_2-\|z(t)-\bar{z}(t)\|_2}{h} \\
            &= \lim_{h\to0^+} \frac{\|z(t)-\bar{z}(t)+h(\dot{z}(t)-\dot{\bar{z}}(t))\|_2-\|z(t)-\bar{z}(t)\|_2}{h} \\
            &= \lim_{h\to0^+} \frac{\|z(t)-\bar{z}(t)+h(\sigma(Az(t)+b)-\sigma(\bar{A}\bar{z}(t)+b))\|_2-\|z(t)-\bar{z}(t)\|_2}{h},
        \end{split}
    \end{align}
    where in the second equality we have used the first order Taylor expansion of $z(t+h)$ and of $\bar{z}(t+h)$ and in the last equality we have used the expressions of the derivatives given by the differential equations \eqref{eq:orig_ode} and \eqref{eq:pert_ode}. We omit from here on the dependence on $t$, and we focus on the term $\sigma(Az+b)-\sigma(\bar{A}z+b)$. Lagrange Theorem ensures the existence of $\xi\in\R^d$, a point on the segment joining $Az+b$ and $\bar{A}\bar{z}+b$, such that
    \[
        \sigma(Az+b)-\sigma(\bar{A}\bar{z}+b) = J_\sigma(\xi)(Az+b-\bar{A}\bar{z}-b) = D(Az-\bar{A}\bar{z}),
    \]
    where $J_\sigma(\xi)$ is the Jacobian matrix of $\sigma$ evaluated at $\xi$, denoted by $D$. We notice that $D\in\Omega_\alpha$. Since $\bar{A}=A+\Delta$, we have that
    \begin{align}\label{eq:geq1}
        \begin{split}
            \sigma(Az+b)-\sigma(\bar{A}\bar{z}+b) &= D(Az-\bar{A}\bar{z}) = D(Az-(A+\Delta)\bar{z}) = DA(z-\bar{z}) - D\Delta \bar{z} \\
            &= DA(z-\bar{z}) + D\Delta (z-\bar{z}) - D\Delta z = D\bar{A}(z-\bar{z}) - D\Delta z.
        \end{split}
    \end{align}
    Therefore, using \eqref{eq:geq1}, Lemma \ref{lem:svd}, and the triangle inequality, we get that
    \begin{align*}
        \dv{}{t^+} \|z-\bar{z}\|_2 &= \lim_{h\to0^+} \frac{\|z-\bar{z}+h(D\bar{A}(z-\bar{z}) - D\Delta z)\|_2-\|z-\bar{z}\|_2}{h} \\
        &\le \lim_{h\to0^+} \frac{\|I+hD\bar{A}\|_2\|z-\bar{z}\|_2 + h\|D\|_2\|\Delta z\|_2-\|z-\bar{z}\|_2}{h} \\
        &\le \lim_{h\to0^+} \frac{\|I+hD\bar{A}\|_2-1}{h} \|z-\bar{z}\|_2 + \|D\|_2 \sigma_{\max}(\Delta) \|z\|_2.
    \end{align*}
    We recall that $\delta=\max_{D\in\Omega_\alpha}\mu_2(D\bar{A})$ and that $\|D\|_2 \le 1$ by definition $\eqref{eq:omega}$ of $\Omega_\alpha$, and we define $M(x)=\max_{t\in[0,1]}\|\phi_t(A_1x+b_1)\|_2$, with $z(t)=\phi_t(A_1x+b_1)$. Notice that $M$ depends on the fixed $x\in K$ and, when clear, we omit the argument $x$. Then
    \begin{align*}
        \dv{}{t^+} \|z-\bar{z}\|_2 &\le \lim_{h\to0^+} \frac{\|I+hD\bar{A}\|_2-1}{h} \|z-\bar{z}\|_2 + \|D\|_2 \sigma_{\max}(\Delta) \|z\|_2 \\
        &\le \mu_2(D\bar{A})\|z-\bar{z}\|_2 + \sigma_{\max}(\Delta) M \le \max_{D\in\Omega_\alpha} \mu_2(D\bar{A})\|z-\bar{z}\|_2 + \sigma_{\max}(\Delta) M \\
        &= \delta \|z-\bar{z}\|_2 + \sigma_{\max}(\Delta) M,
    \end{align*}
    with $\sigma_{\max}(\Delta)$ the largest singular value of $\Delta$. Applying Lemma \ref{lem:gronwall} with $u(t)=\|z(t)-\bar{z}(t)\|_2$, $t\in[0,1]$, yields
    \[
        \|z(t)-\bar{z}(t)\|_2 \le e^{\delta t} \|z(0)-\bar{z}(0)\|_2 + \frac{\sigma_{\max}(\Delta) M}{\delta}(e^{\delta t}-1) = \sigma_{\max}(\Delta)M\frac{e^{\delta t}-1}{\delta},
    \]
    since $z(0)=\bar{z}(0)$ (see \eqref{eq:orig_ode} and \eqref{eq:pert_ode}). In particular, if $t=1$, we have that
    \begin{equation}\label{eq:seco_ineq}
        \| \phi(A_1x+b_1)-\bar{\phi}(A_1x+b_1) \|_2 = \|z(1)-\bar{z}(1)\|_2 \le M \sigma_{\max}(\Delta) \frac{e^{\delta}-1}{\delta}.
    \end{equation}
    By plugging \eqref{eq:seco_ineq} into \eqref{eq:main_ineq}, we obtain that
    \[
        \|\varphi(x)-\bar{\varphi}(x)\|_2 \le M \sigma_{\max}(\Delta) \frac{e^{\delta}-1}{\delta}.
    \]
    If we define $C:=\max_{x\in K}M(x)$, then
    \[
        \|f-\bar{\varphi}\|_{\infty,K} = \max_{x\in K} \|f(x)-\bar{\varphi}(x)\|_2 \le C \sigma_{\max}(\Delta) \frac{e^{\delta}-1}{\delta} + \varepsilon,
    \]
    which concludes the proof.
\end{proof}

We notice that the proposed upper bound is pointwise in the space of functions, in the sense that it depends on the choice of $\varphi$. Specifically, given a function $f\in \cC(\R^m,\R^n)$ and a shallow neural network $\varphi\in\mathcal{H}_{c,1}$ that approximates within $\varepsilon$ the target function $f$, the upper bound depends on the value $\delta$ and the perturbation matrix $\Delta$.
\begin{itemize}
    \item $e^{\delta}$ is the upper bound to the Lipschitz constant of the stabilised shallow neural network $\bar{\varphi}$.
    \item $\Delta$ is the perturbation computed to define the stabilised shallow neural network $\bar{\varphi}$. The larger is $\sigma_{\max}(\Delta)$, the larger is the upper bound, and thus the worse the approximation error may be.
\end{itemize}
\begin{remark}
The result above does not necessarily imply that a Lipschitz-constrained neural network $\bar{\varphi}$ cannot accurately approximate any given target function. It simply tells us that if we start from a trained, unconstrained neural network and correct its weights as illustrated above, we can expect a controlled drop in accuracy.
\end{remark}
\begin{corollary}
    With the same notations and under the assumptions of Theorem \ref{th:upp_bound}, if $\delta=0$, i.e. $\bar{\varphi}$ is a nonexpansive shallow neural network, then the approximation upper bound becomes
    \[
        \|f-\bar{\varphi}\|_{\infty,K}  \le C \sigma_{\max}(\Delta) + \varepsilon.
    \]
\end{corollary}

We now move to the approximation lower bound.

\subsection{Approximation lower bound}

Some preliminaries follow. For a shallow neural network $\varphi\in\mathcal{H}_{c,1}$, we define $\delta_\star := \max_{D\in\Omega_\alpha} \mu_2(DA)$ and, for a fixed $\delta<\delta_\star$, we recall that there exists a perturbation matrix $\Delta$ such that
\[
    \max_{D\in\Omega_\alpha} \mu_2(D(A+\Delta)) = \delta.
\]
We then define the perturbed matrix $\bar{A}$ and the stabilised shallow neural network $\bar{\varphi}$:
\begin{enumerate}
    \item $\bar{A} = A + \Delta$,
    \item $\delta' = - \max_{D\in\Omega_\alpha} \mu_2(-D\bar{A})$,
    \item $\bar{\varphi}(x) = A_2\bar{\phi}(A_1x+b_1)+b_2, \quad x\in\R^m$.
\end{enumerate}

\begin{definition}
    If $u$ and $v$ are real vectors, the cosine of the angle $\theta(u,v)$ between the vectors $u$ and $v$ is defined as
    \[
        \cos\theta(u,v)=\frac{\langle u,v \rangle}{\|u\|_2\|v\|_2}.
    \]
\end{definition}

\begin{assumption}\label{ass:ass}
    Given a compact subset $K$ of $\R^m$ and $0<\bar{t}\ll1$, we assume for all $x\in K$ that
    \begin{equation*}
        \min_{D\in\Omega_\alpha,\ t\in[\bar{t},1]}\cos\theta(\phi_t(A_1x+b_1)-\bar{\phi}_t(A_1x+b_1),-D\Delta \phi_t(A_1x+b_1))>0,
    \end{equation*}
    and that the cosine is well defined, where all quantities in the inequality are defined and introduced in the beginning of this section.
\end{assumption}

\begin{theorem}
    We consider a function $f\in \cC(\R^m,\R^n)$, a compact subset $K$ of $\R^m$, and a shallow neural network $\varphi\in\mathcal{H}_{c,1}$ such that
    \[
        \|f-\varphi\|_{\infty,K} \le \varepsilon
    \]
    for some $\varepsilon>0$. Furthermore, we suppose that, for a fixed $0<\bar{t}\ll1$, Assumption \ref{ass:ass} holds. Then, there exists $c_1(\bar{t})>0$ and $c_2>0$ such that
    \[
        \|f-\bar{\varphi}\|_{\infty,K} \ge c_1(\bar{t})e^{\delta'(1-\bar{t})} + c_2 \sigma_{\min}(\Delta) \frac{e^{\delta'(1-\bar{t})} - 1}{\delta'} - \varepsilon.
    \]
\end{theorem}

\begin{proof}
    We recall that
    \[
        \|f-\bar{\varphi}\|_{\infty,K} = \max_{x\in K} \|f(x)-\bar{\varphi}(x)\|_2.
    \]
    For a fixed $x\in K$, by the reverse triangle inequality, we get
    \begin{align*}
        \|f(x)-\bar{\varphi}(x)\|_2 &= \|f(x)-\varphi(x)+\varphi(x)-\bar{\varphi}(x)\|_2 = \|\varphi(x)-\bar{\varphi}(x) - (\varphi(x)-f(x))\|_2 \\
        &\ge | \|\varphi(x)-\bar{\varphi}(x)\|_2 - \|\varphi(x)-f(x)\|_2 | \ge \|\varphi(x)-\bar{\varphi}(x)\|_2 - \varepsilon,
    \end{align*}
    so we need to compute a lower bound for $\|\varphi(x)-\bar{\varphi}(x)\|_2$. Then, by definition of $\varphi$ and $\bar{\varphi}$, and using Lemma \ref{lem:svd}, we have that
    \begin{align}\label{eq:main_ineq_2}
        \begin{split}
        \|\varphi(x)-\bar{\varphi}(x)\|_2 &= \|A_2\phi(A_1x+b_1)+b_2-A_2\bar{\phi}(A_1x+b_1)-b_2\|_2 \\
        &= \|A_2(\phi(A_1x+b_1)-\bar{\phi}(A_1x+b_1))\|_2 \\
        &\ge \sigma_{\min}(A_2) \| \phi(A_1x+b_1)-\bar{\phi}(A_1x+b_1) \|_2,
        \end{split}
    \end{align}
    with $\sigma_{\min}(A_2)$ the smallest singular value of $A_2$. Therefore we need to compute a lower bound for the difference of the flow maps of two differential equations from the same initial datum.
    
    We define $z(t)=\phi_t(A_1x+b_1)$ and $\bar{z}(t)=\bar{\phi}_t(A_1x+b_1)$ for all $t\in[0,1]$ as in \eqref{eq:orig_ode} and \eqref{eq:pert_ode}, with $\phi_1=\phi$ and $\bar{\phi}_1 = \bar{\phi}$. By Assumption \ref{ass:ass}, it holds that $\|z(t)-\bar{z}(t)\|_2>0$ for all $t\in[\bar{t},1]$. Then, omitting the dependence on $t$ and using \eqref{eq:geq1} and Lemma \ref{lem:l2n}, we have that
    \begin{align*}
        \dv{}{t} \|z-\bar{z}\|_2 &= \frac{\langle z-\bar{z} , \dot{z}-\dot{\bar{z}} \rangle}{\|z-\bar{z}\|_2} = \frac{\langle z-\bar{z} , \sigma(Az+b)-\sigma(\bar{A}\bar{z}+b) \rangle}{\|z-\bar{z}\|_2} \\
        &= \frac{\langle z-\bar{z} , D\bar{A}(z-\bar{z})-D\Delta z \rangle}{\|z-\bar{z}\|_2} = \frac{\langle z-\bar{z} , D\bar{A}(z-\bar{z})\rangle}{\|z-\bar{z}\|_2} + \frac{\langle z-\bar{z} , - D\Delta z \rangle}{\|z-\bar{z}\|_2} \\
        &\ge -\mu_2(-D\bar{A}) \|z-\bar{z}\|_2 + \|D\Delta z\|_2 \cos\theta(z-\bar{z},-D\Delta z),
    \end{align*}
    where $\theta(z-\bar{z},-D\Delta z)$ denotes the angle between $z-\bar{z}$ and $-D\Delta z$. We recall that $\delta'=-\max_{D\in\Omega_\alpha}\mu_2(-D\bar{A})$ and we define $\eta(x):=\min_{D\in\Omega_\alpha,\ t\in[\bar{t},1]}\cos\theta(z(t)-\bar{z}(t),-D\Delta z(t))$. Notice that $\eta$ depends on the fixed $x\in K$ and, when clear, we omit the argument $x$. Thus, we get that
    \begin{equation}\label{eq:step}
        \dv{}{t} \|z-\bar{z}\|_2 \ge \delta' \|z-\bar{z}\|_2 + \|D\Delta z\|_2 \eta.
    \end{equation}
    Since $z(t)=\phi_t(A_1x+b_1)$ and $\bar{z}(t)=\bar{\phi}_t(A_1x+b_1)$, then $\eta>0$ by assumption \ref{ass:ass}. Thus, applying Lemma \ref{lem:svd} and using the fact that $D\in\Omega_\alpha$, we obtain that
    \[
        \|D\Delta z\|_2 = \sqrt{\sum_{i=1}^d D_{ii}^2 (\Delta z)_{ii}^2} \ge \alpha \sqrt{\sum_{i=1}^d (\Delta z)_{ii}^2} = \alpha \|\Delta z\|_2 \ge \alpha \sigma_{\min} (\Delta) \|z\|_2,
    \]
    and then
    \[
        \dv{}{t} \|z-\bar{z}\|_2 \ge \delta' \|z-\bar{z}\|_2 + \alpha \eta \sigma_{\min} (\Delta) m,
    \]
    where $m(x):=\min_{t\in[\bar{t},1]}\|\phi_t(A_1x+b_1)\|_2$, with $z(t)=\phi_t(A_1x+b_1)$. Notice that $m$ depends on the fixed $x\in K$ and, when clear, we omit the argument $x$. Applying Lemma \ref{lem:gronwall} with $u(t)=\|z(t)-\bar{z}(t)\|_2$, $t\in[\bar{t},1]$, $a:=\delta'$ and $b:=\alpha\eta\sigma_{\min}(\Delta)m$, yields
    \begin{align*}
        \|z(t)-\bar{z}(t)\|_2 &\ge e^{\delta'(t-\bar{t})} \|z(\bar{t})-\bar{z}(\bar{t})\|_2 + \frac{\alpha\eta\sigma_{\min}(\Delta)m}{\delta'}(e^{\delta' (t-\bar{t})}-1) \\
        &=e^{\delta' (t-\bar{t})} \|\phi_{\bar{t}}(A_1x+b_1)-\bar{\phi}_{\bar{t}}(A_1x+b_1)\|_2 + \alpha\eta\sigma_{\min}(\Delta)m\frac{e^{\delta' (t-\bar{t})}-1}{\delta'}.
    \end{align*}
    In particular, if $t=1$, we have that
    \begin{align}\label{eq:third_ineq}
        \begin{split}
            &\| \phi(A_1x+b_1)-\bar{\phi}(A_1x+b_1) \|_2 = \|z(1)-\bar{z}(1)\|_2 \\
            &\ge e^{\delta'(1-\bar{t})} \|\phi_{\bar{t}}(A_1x+b_1)-\bar{\phi}_{\bar{t}}(A_1x+b_1)\|_2 + \alpha\eta\sigma_{\min}(\Delta)m\frac{e^{\delta'(1-\bar{t})}-1}{\delta'}.
        \end{split}
    \end{align}
    By plugging \eqref{eq:third_ineq} into \eqref{eq:main_ineq_2}, we obtain that
    \begin{align*}
        \|\varphi(x)-\bar{\varphi}(x)\|_2 &\ge \sigma_{\min}(A_2) \|\phi_{\bar{t}}(A_1x+b_1)-\bar{\phi}_{\bar{t}}(A_1x+b_1)\|_2 e^{\delta' (1-\bar{t})} \\ &\quad+ \alpha\eta\sigma_{\min}(A_2)\sigma_{\min}(\Delta)m\frac{e^{\delta'(1-\bar{t})}-1}{\delta'}.
    \end{align*}
    If we define
    \begin{align*}
        c_1(\bar{t})&:=\sigma_{\min}(A_2)\min_{x\in K}\|\phi_{\bar{t}}(A_1x+b_1)-\bar{\phi}_{\bar{t}}(A_1x+b_1)\|_2, \\
        c_2&:=\alpha\sigma_{\min}(A_2)\left(\min_{x\in K}m(x)\right)\left(\min_{x\in K}\eta(x)\right),
    \end{align*}
    then
    \[
        \|f(x)-\bar{\varphi}(x)\|_2 \ge c_1(\bar{t}) e^{\delta' (1-\bar{t})} + c_2 \sigma_{\min}(\Delta)\frac{e^{\delta'(1-\bar{t})}-1}{\delta'} - \varepsilon,
    \]
    and so
    \[
        \|f-\bar{\varphi}\|_{\infty,K} \ge c_1(\bar{t}) e^{\delta'(1-\bar{t})} + c_2 \sigma_{\min}(\Delta)\frac{e^{\delta'(1-\bar{t})}-1}{\delta'} - \varepsilon,
    \]
    which concludes the proof.
\end{proof}

\begin{remark}
    How would the proof of the approximation lower bound change if
    \[
        \min_{D\in\Omega_\alpha,\ t\in[\bar{t},1]}\cos\theta(\phi_t(A_1x+b_1)-\bar{\phi}_t(A_1x+b_1),-D\Delta \phi_t(A_1x+b_1))\le0?
    \]
    Then
    \[
        \eta:=\min_{D\in\Omega_\alpha,\ t\in[\bar{t},1]}\cos\theta(z(t)-\bar{z}(t),-D\Delta z(t))\le0
    \]
    and, since,
    \[
        \|D\Delta z\|_2 = \sqrt{\sum_{i=1}^d D_{ii}^2 (\Delta z)_{ii}^2} \le \sqrt{\sum_{i=1}^d (\Delta z)_{ii}^2} = \|\Delta z\|_2 \le \sigma_{\max} (\Delta) \|z\|_2,
    \]
    it follows from Equation \eqref{eq:step} that
    \[
        \dv{}{t} \|z-\bar{z}\|_2 \ge \delta' \|z-\bar{z}\|_2 + \eta \sigma_{\max} (\Delta) M,
    \]
    where $M(x):=\max_{t\in[\bar{t},1]}\|\phi_t(A_1x+b_1)\|_2$. By following the same reasoning of the proof of the approximation lower bound, if we define $c_1(\bar{t}):=\sigma_{\min}(A_2)\min_{x\in K}\|\phi_{\bar{t}}(A_1x+b_1)-\bar{\phi}_{\bar{t}}(A_1x+b_1)\|_2$ and $c_2:=\sigma_{\min}(A_2)\left(\max_{x\in K}M(x)\right)\left(\min_{x\in K}\eta(x)\right)$, we obtain that
    \[
        \|f-\bar{\varphi}\|_{\infty,K} \ge c_1(\bar{t}) e^{\delta'(1-\bar{t})} + c_2 \sigma_{\max}(\Delta)\frac{e^{\delta'(1-\bar{t})}-1}{\delta'} - \varepsilon.
    \]
    However $c_2<0$ here, so the approximation lower bound could be negative and therefore useless.
\end{remark}

\begin{remark}
    The presented approximation bounds are valid if the shallow neural networks \eqref{eq:snn} and \eqref{eq:ssnn} have the same parameters $b,A_1,b_1,A_2,b_2$. However, back to the MNIST classification problem, if we further train the parameters $b,A_1,b_1,A_2,b_2$ of
    \[
        \bar{\varphi}(x) = A_2\bar{\phi}(A_1x+b_1)+b_2, \quad x\in\R^m,
    \]
    keeping the activation function $\bar{\phi}$ fixed and fulfilling the constraint $\|A_1\|_2\ge1$ a fixed constant and $\|A_2\|_2=1$, we obtain the shallow neural network
    \[
        \bar{\varphi}_\star(x) = \bar{A}_2 \bar{\phi} (\bar{A}_1 x + \bar{b}_1) + \bar{b}_2,
    \]
    with $\|\bar{A}_1\|_2>1$ a fixed constant and $\|\bar{A}_2\|_2=1$, that is more accurate than $\bar{\varphi}$. See \eqref{eq:snn}, \eqref{eq:nODE}, and \eqref{eq:ssnn} for the notations. Indeed, the two models have the same Lipschitz constant since $\bar{\phi}$ has not changed, but $\bar{\varphi}_\star(x)$ has better accuracy as shown in Table \ref{tab:acc2}. This indicates that the corrected neural ODE with flow map $\bar{\phi}$ is not necessarily the best neural ODE with Lipschitz constant $e^\delta$. This does not contradict the result presented above, but it suggests that the approximation bounds can be improved by allowing the parameters $b,A_1,b_1,A_2,b_2$ to change.
    \begin{table}[ht]
        \centering
        \begin{tabular}{ c c c c c c c c }
            \toprule
             & $\delta=-2$ & $\delta=-1$ & $\delta=0$ & $\delta=1$ & $\delta=2$ & $\delta=3$ & $\delta=4$ \\
            \midrule
            accuracy of $\bar{\varphi}$ & 0.1006 & 0.1093 & 0.5099 & 0.9274 & 0.9623 & 0.9736 & 0.9762 \\
            accuracy of $\bar{\varphi}_\star$ & 0.9029 & 0.9314 & 0.9683 & 0.9765 & 0.9767 & 0.9761 & 0.9768 \\
            \bottomrule
        \end{tabular}
        \caption{Accuracy of $\bar{\varphi}$ and $\bar{\varphi}_\star$ for different values of $\delta$.}
        \label{tab:acc2}
    \end{table}
\end{remark}

\begin{remark}
    Not only the shallow neural network $\bar{\varphi}_\star$ has accuracy approaching the accuracy of $\varphi$ as $\delta$ approaches $\delta_\star$, see Table \ref{tab:acc2}, but it is more stable as expected. Back to the MNIST classification problem, we observe the behaviour of the model $\bar{\varphi}_\star$ for different values of the parameter $\delta$. In Figure \ref{fig:comp1}, we plot the accuracy of $\bar{\varphi}_\star$ for different values of $\delta$ on the MNIST dataset as a function of the magnitude $\eta$ of the FGSM adversarial attack. As $\delta$ increases, the model $\bar{\varphi}_\star$ gains in accuracy, but it loses in robustness and stability. In contrast, as $\delta$ decreases, the model's robustness and stability improve, but its accuracy deteriorates. See \cite[Section 5]{de2025stability} for more details.

    \begin{figure}[ht]
        \centering
        \includegraphics[width=0.85\textwidth]{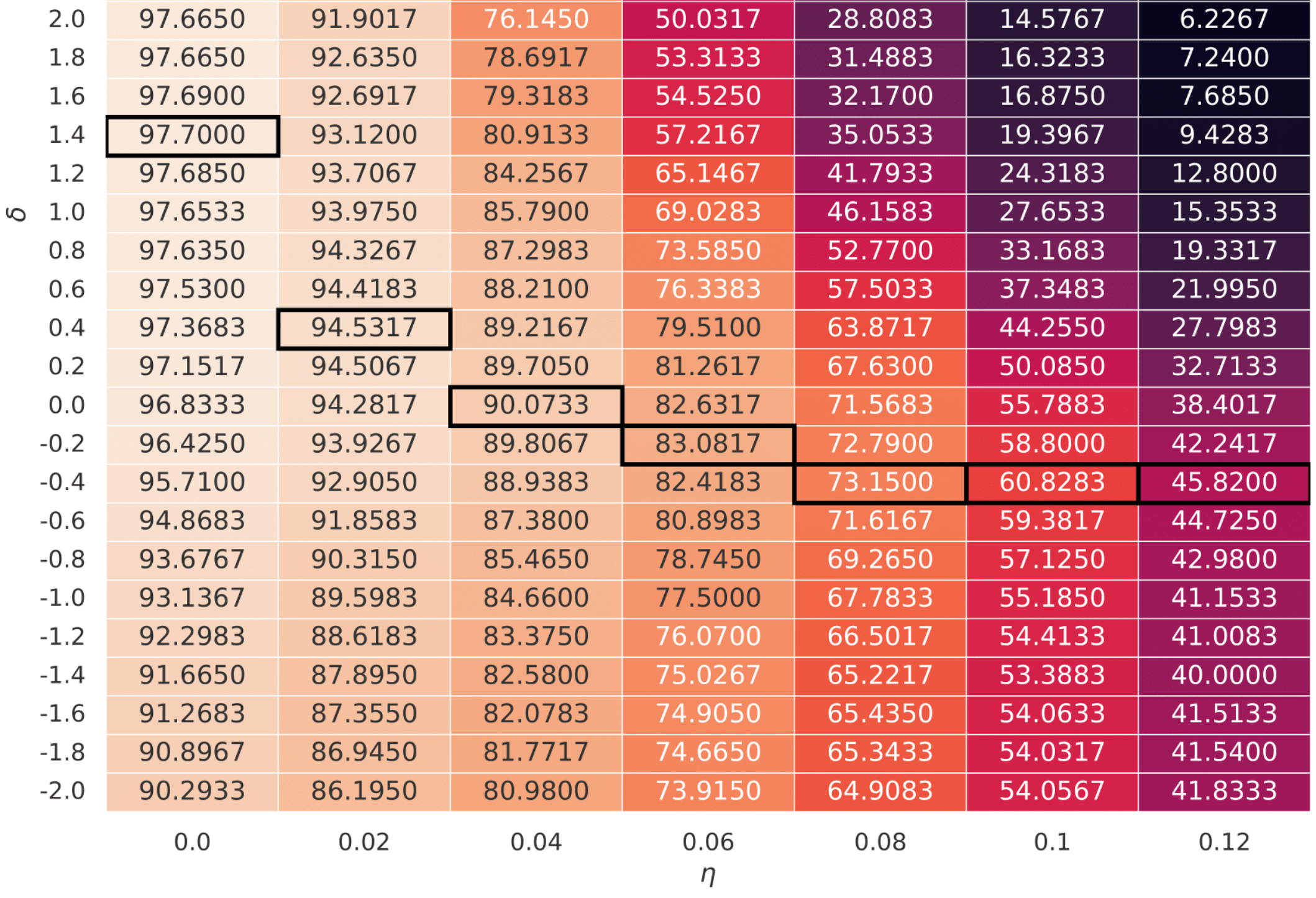}
        \caption{Accuracy of $\bar{\varphi}_\star$ for different values of $\delta$ as a function of the magnitude $\eta$ of the FGSM adversarial attack on the MNIST dataset. For each perturbation magnitude $\eta$, the best validation accuracy is highlighted.}
        \label{fig:comp1}
    \end{figure}
\end{remark}

\begin{remark}
    We notice that the approximation lower bound does not hold for any compact set $K$. Indeed, it holds for any compact set $K$ such that for all $x\in K$ Assumption \ref{ass:ass} holds, i.e.
    \[
        \min_{D\in\Omega_\alpha,\ t\in[\bar{t},1]}\cos\theta(\phi_t(A_1x+b_1)-\bar{\phi}_t(A_1x+b_1),-D\Delta \phi_t(A_1x+b_1))>0.
    \]
    On the other hand, given a compact set $K$, such assumption can be used to detect those regions of a compact set $K$ where Assumption \ref{ass:ass}, and therefore the lower bound holds. We clarify this with two examples.
\end{remark}

\subsection{Example 1}
We consider the function
\[
    f(x) = \phi(x), \quad x\in[-1,1]^2,
\]
where $\phi:[-1,1]^2\to\R^2$ is the flow map at time 1 of the neural ODE
\[
    \dot{u}(t) = \sigma(Au(t)+b), \quad t\in[0,1].
\]
We choose $\sigma:\R\to\R$ to be the LeakyReLU with minimal slope $\alpha=0.1$, i.e. $\sigma(x)=\max\{x,\alpha x\}$, and we randomly draw ten times the parameters $A\in\R^{2\times2}$ and $b\in\R^2$ from a uniform distribution on the interval $[0,1)$ and from a standard normal distribution. In this setting,
\[
    \Omega_\alpha = \left\{
    \begin{bmatrix}
        \alpha & 0 \\
             0 & \alpha
    \end{bmatrix},
    \begin{bmatrix}
        \alpha & 0 \\
             0 & 1
    \end{bmatrix},
    \begin{bmatrix}
             1 & 0 \\
             0 & \alpha
    \end{bmatrix},
    \begin{bmatrix}
             1 & 0 \\
             0 & 1
    \end{bmatrix}
    \right\},
\]
and, following the approach proposed in \cite{guglielmi2024contractivity}, we compute $\delta_\star = \max_{D\in\Omega_\alpha} \mu_2(DA)$ for each random choice of $A$. Next, given $\delta \in \{ \delta_\star - i \ : \ i\in\{0.01,\ 0.02,\ 0.03,\ 0.04,\ 0.05,\ 0.06,\ 0.07,\ 0.08,\ 0.09\} \}$, we compute the smallest (in Frobenius norm) perturbation matrix $\Delta\in\R^{2\times2}$ as in \cite{de2025stability} to each random choice of $A$ such that
\[
    \max_{D\in\Omega_\alpha} \mu_2(D(A+\Delta)) = \delta,
\]
and we denote by $\bar{\phi}:[-1,1]^2\to\R^2$ the flow map of perturbed neural ODE
\[
    \dot{u}(t) = \sigma((A+\Delta)u(t)+b), \quad t\in[0,1].
\]
We now provide an answer to the following question: where does the lower bound for the approximation error
\[
    \|f-\bar{\phi}\|_{\infty,[-1,1]^2} = \|\phi-\bar{\phi}\|_{\infty,[-1,1]^2}
\]
hold in the domain $[-1,1]^2$? It is sufficient to check where Assumption \ref{ass:ass} holds in $[-1,1]^2$. In this setting, Assumption \ref{ass:ass} is simply
\[
    \min_{D\in\Omega_\alpha,\ t\in[\bar{t},1]}\cos\theta(\phi_t(x)-\bar{\phi}_t(x),-D\Delta \phi_t(x))>0,
\]
since $A_1=I_2$ and $b_1=0$.

We discretise the domain $K=[-1,1]^2$ with constant stepsize $h=0.05$ in both directions and we call $K_h$ the discretised domain. Then, we fix $\bar{t}=0.3$, so that we are safely far from 0, and we introduce a discretisation $I_k$ of the time interval $I = [\bar{t},1]$ with stepsize $k=0.05$. We compute the flow maps $\phi_t(x)$ and $\bar{\phi}_t(x)$ numerically using the explicit Euler method with constant stepsize 0.05, and we compute for each $x \in K_h$
\[
    \eta(x) = \min_{D\in\Omega_\alpha,\ t\in I_k}\cos\theta(\phi_t(x)-\bar{\phi}_t(x),-D\Delta \phi_t(x)).
\]
If $\eta(x)$ is positive, then the lower bound holds, otherwise it does not. We have detected in this way those regions of the discretised domain $K_h$ where the lower bound holds. It is interesting to notice that, as $\delta$ decreases, the regions of the discretised domain $K_h$, where the lower bound holds, expand.
\begin{figure}[ht]
    \begin{subfigure}{0.5\textwidth}
        \centering
        \includegraphics[width=\textwidth]{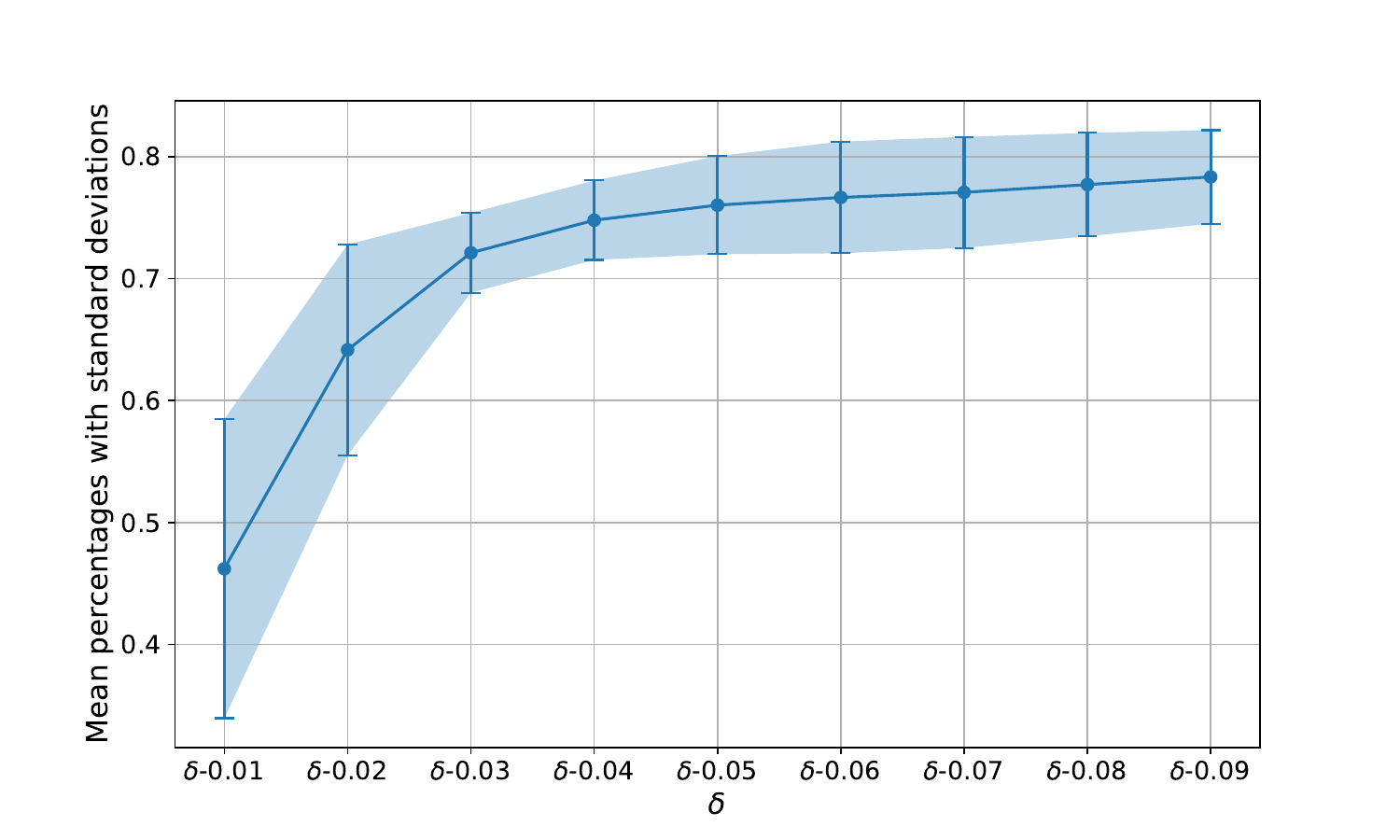}
        \caption{The parameters $A$ and $b$ are drawn from an uniform distribution on the interval $[0,1)$.}
        \label{fig:rand}
    \end{subfigure}
    \begin{subfigure}{0.5\textwidth}
        \centering
        \includegraphics[width=\textwidth]{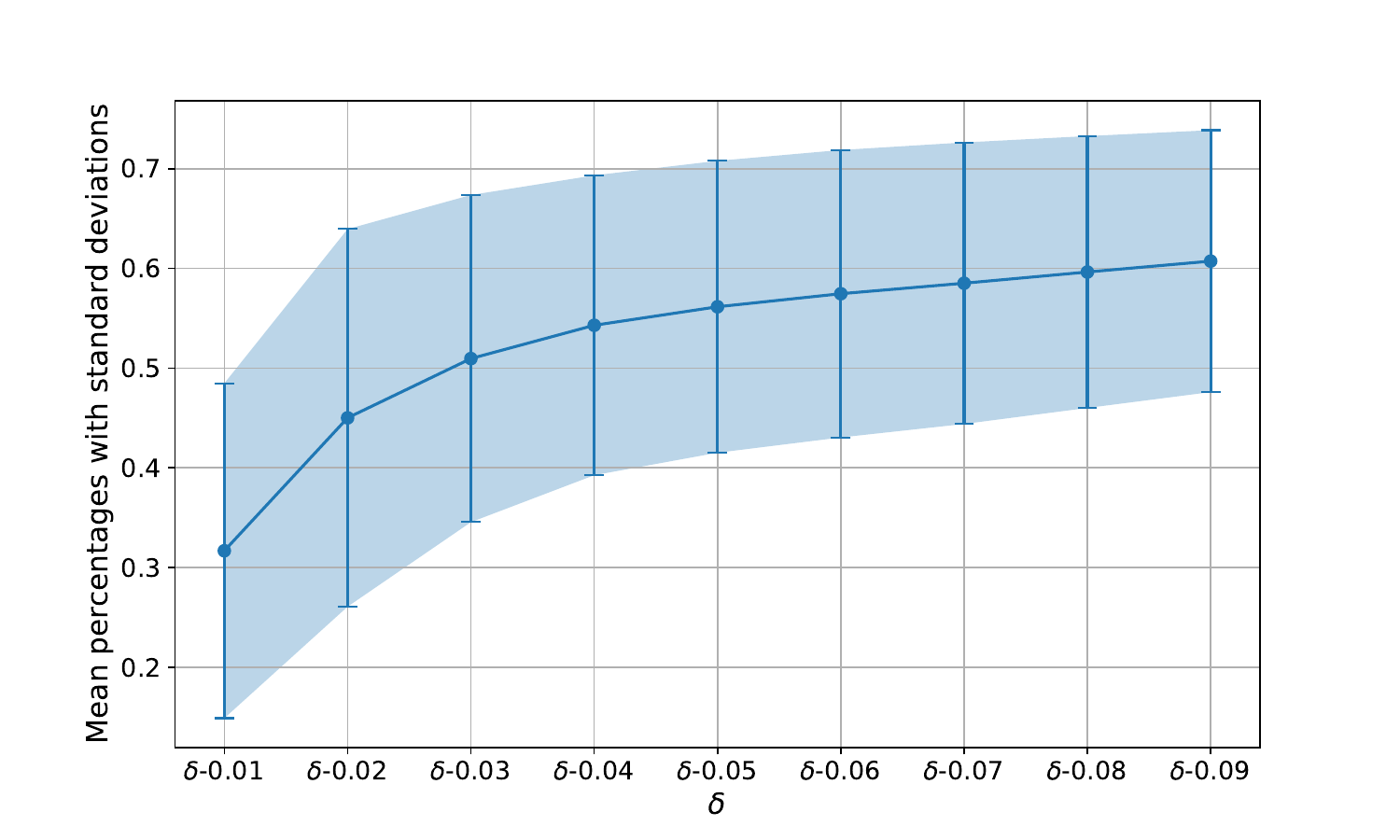}
        \caption{The parameters $A$ and $b$ are drawn from a standard normal distribution.}
        \label{fig:randn}
    \end{subfigure}
    \caption{Mean and standard deviation of the percentage of points of the discretised domain $K_h$, where the lower bound holds, with respect to the random choice of parameters $A$ and $b$ as a function of $\delta$.}
\end{figure}
This is reasonable since, as $\delta$ decreases, the functions $\phi$ and $\bar{\phi}$ differ more and more. Indeed, for each value of $\delta$, we compute the mean and the standard deviation of the percentage of points of the discretised domain $K_h$, where the lower bound holds, with respect to the random choice of parameters $A$ and $b$. In Figure \ref{fig:rand} we plot the mean and the standard deviation with respect to $\delta$ in the case the parameters $A$ and $b$ are drawn randomly from a uniform distribution on the interval $[0,1)$. In Figure \ref{fig:randn} we plot the same in the case the parameters $A$ and $b$ are drawn randomly from a standard normal distribution.

\subsection{Example 2}
We consider the shallow neural network
\[
    \varphi(x) = A_2 \phi( A_1 x + b_1 ) + b_2, \quad x\in\R^2,
\]
where $A_1\in\R^{4\times2}$, $b_1\in\R^4$, $A_2\in\R^{2\times4}$, $b_2\in\R^2$ and $\phi:\R^4\to\R^4$ is the flow map at time 1 of the neural ODE
\[
    \dot{u}(t) = \sigma(Au(t)+b), \quad t\in[0,1],
\]
with $A\in\R^{4\times4}$ and $b\in\R^4$. We choose $\sigma:\R\to\R$ to be the LeakyReLU with minimal slope $\alpha=0.1$, i.e. $\sigma(x)=\max\{x,\alpha x\}$, and we train $\varphi$ on the Two Moons dataset. The obtained accuracy is 100\%. See Figure \ref{fig:tmd}.
\begin{figure}[ht]
    \centering
    \includegraphics[width=0.5\textwidth]{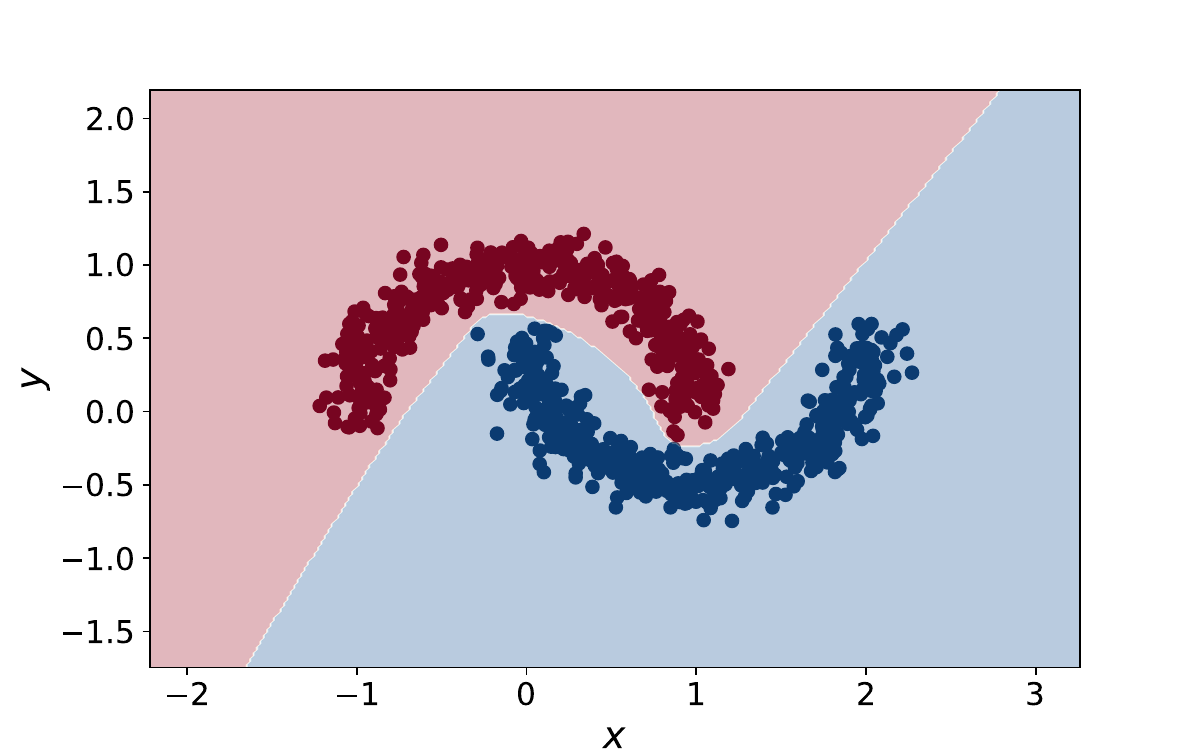}
    \caption{Two Moons dataset.}
    \label{fig:tmd}
\end{figure}
In this setting, $\Omega_\alpha$ has cardinality $2^4=16$ and, for the trained weight matrix $A$, following the approach proposed in \cite{guglielmi2024contractivity}, it turns out that (see Definition \ref{def:log_norm})
\[
    \delta_\star = \max_{D\in\Omega_\alpha} \mu_2(DA) = 8.4521.
\]
Given $\delta \in \{ \delta_\star - 0.006 ,\ \delta_\star - 0.004 ,\ \delta_\star - 0.002 \}$, we compute the smallest (in Frobenius norm) perturbation matrix $\Delta\in\R^{4\times4}$ as in \cite{de2025stability} such that
\[
    \max_{D\in\Omega_\alpha} \mu_2(D(A+\Delta)) = \delta,
\]
and we denote by $\bar{\phi}:\R^4\to\R^4$ the flow map of perturbed neural ODE
\[
    \dot{u}(t) = \sigma((A+\Delta)u(t)+b), \quad t\in[0,1],
\]
and by $\bar{\varphi}$ the stabilized shallow neural network
\[
    \bar{\varphi}(x) = A_2\bar{\phi}(A_1x+b_1)+b_2, \quad x\in\R^2.
\]
Here the compact subset $K\subset\R^2$ of interest is the Two Moons dataset itself, and it is already a discrete set. We now provide an answer to the following question: where does the lower bound for the approximation error
\[
    \|\varphi-\bar{\varphi}\|_{\infty,K}
\]
hold in the domain $K$? It is sufficient to check where Assumption \ref{ass:ass} holds in $K$. We fix $\bar{t}=0.3$, so that we are safely far from 0, and we introduce a discretisation $I_k$ of the time interval $I = [\bar{t},1]$ with stepsize $k=0.05$. We compute for each $x \in K$
\[
    \eta(x) = \min_{D\in\Omega_\alpha,\ t\in I_k}\cos\theta(\phi_t(A_1x+b_1)-\bar{\phi}_t(A_1x+b_1),-D\Delta \phi_t(A_1x+b_1)),
\]
and we make green those points for which $\eta(x)$ is positive, while we make red those points for which $\eta(x)$ is negative. We have detected in this way those regions of the Two Moons dataset where the lower bound holds. See Figure \ref{fig:tmdlb} for the green region where the lower bound holds. As $\delta$ decreases, we notice that the region, where the lower bound holds, expand. % and Figure \ref{fig:tmdacc} for the different classification obtained for the different values of $\delta$.

About the classification accuracy, we do not observe a loss of accuracy. This is reasonable, since the considered values of $\delta$ are very close to $\delta_\star$, and so the perturbed matrix $A+\Delta$ is very close to the unperturbed matrix $A$. In order to observe a loss of accuracy, we need to further decrease $\delta$. For example, we consider $\delta \in \{ \delta_\star - 2 ,\ \delta_\star - 1.5 ,\ \delta_\star - 1 \}$. Now, the region, where the lower bound holds, is the same in all three cases, but the accuracy is decreasing as $\delta$ decreases. See Figure \ref{fig:tmdlb1} for the green region where the lower bound holds, and Figure \ref{fig:tmdacc1} for the different classification obtained for the different values of $\delta$.

%We notice that the red region where the lower bound is not guaranteed to hold is the same in the three subfigures of Figure \ref{fig:tmdlb}. This is consistent with Figure \ref{fig:tmdacc} because if the lower bound for the approximation error in the red region does not hold, then $\varphi$ and $\bar{\varphi}$ take about the same values in the red region and therefore the stabilised shallow neural network $\bar{\varphi}$ classify the red region correctly.

%We also notice that there are regions of the Two Moons dataset where the lower bound holds and that are correctly classified by the stabilised shallow neural network $\bar{\varphi}$. This happens because we are addressing a classification problem, so $\varphi$ and $\bar{\varphi}$ may take different values on the same region and still classify it correctly. Thus, the existence of the lower bound does not prevent the stabilised shallow neural network $\bar{\varphi}$ from correctly classifying the points in the region. 
\begin{figure}[ht]
    \begin{subfigure}{0.32\textwidth}
        \centering
        \includegraphics[width=\textwidth]{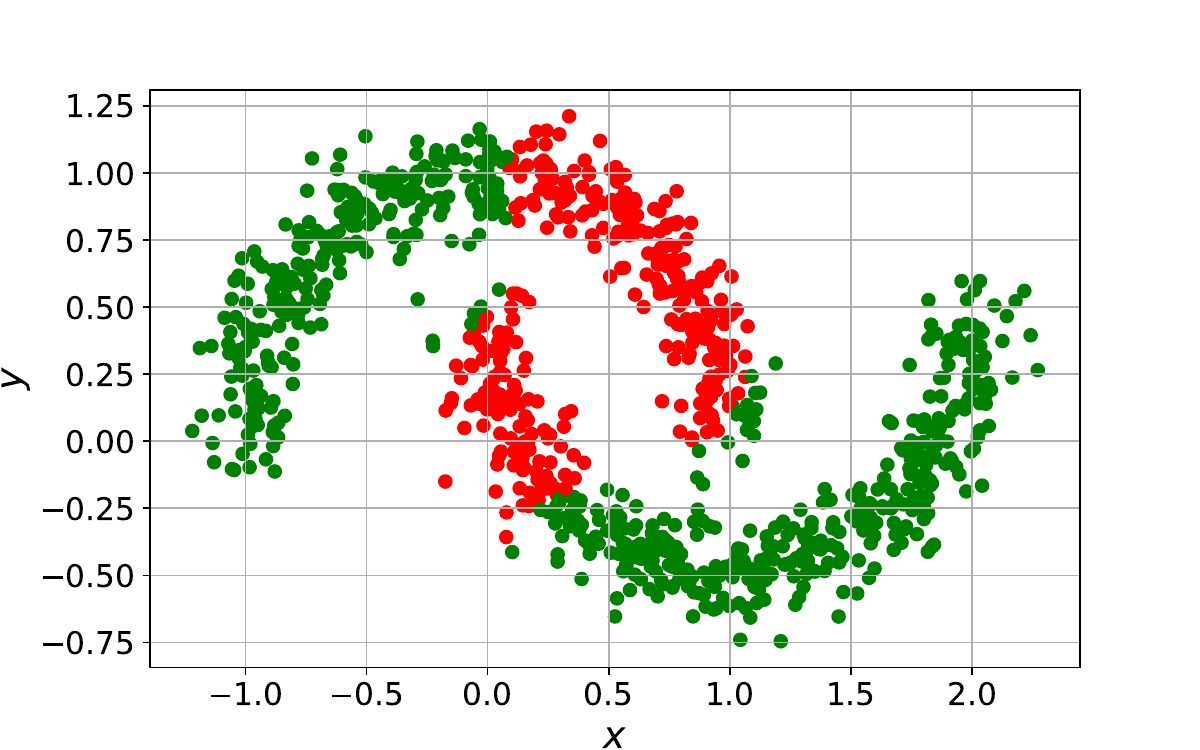}
        \caption{$\delta=\delta_\star-0.006$.}
    \end{subfigure}
    \hfill
    \begin{subfigure}{0.32\textwidth}
        \centering
        \includegraphics[width=\textwidth]{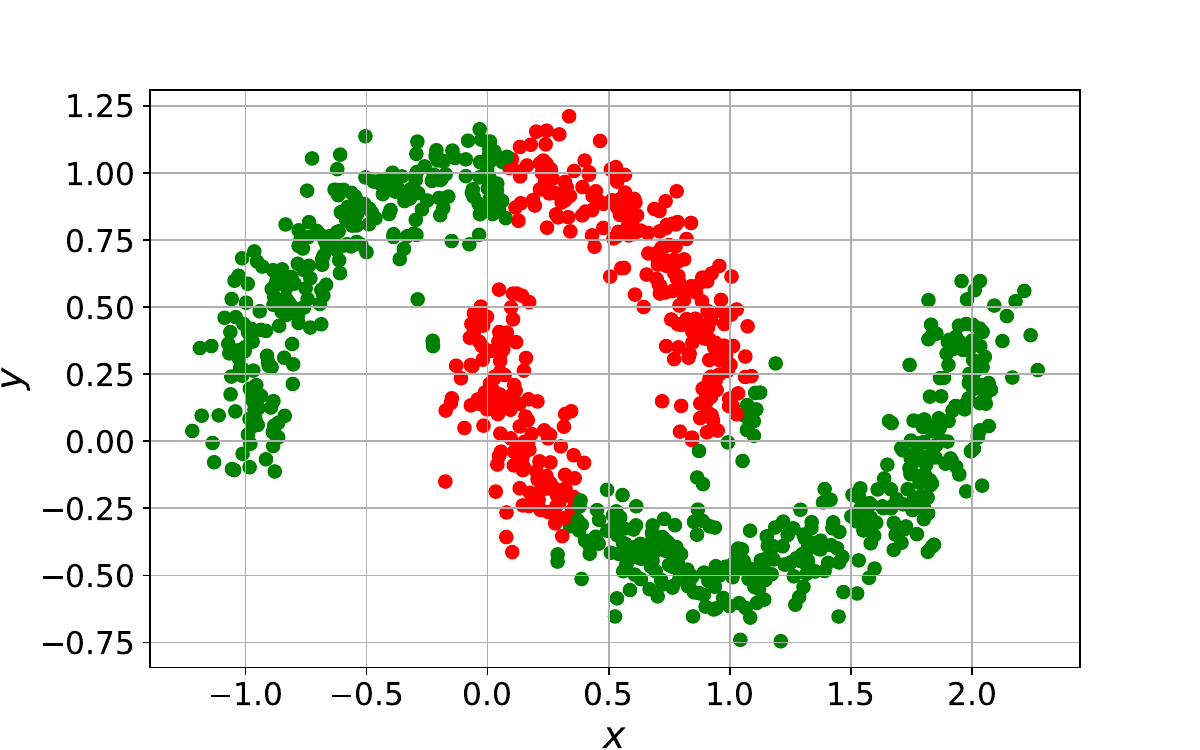}
        \caption{$\delta=\delta_\star-0.004$.}
    \end{subfigure}
    \hfill
    \begin{subfigure}{0.32\textwidth}
        \centering
        \includegraphics[width=\textwidth]{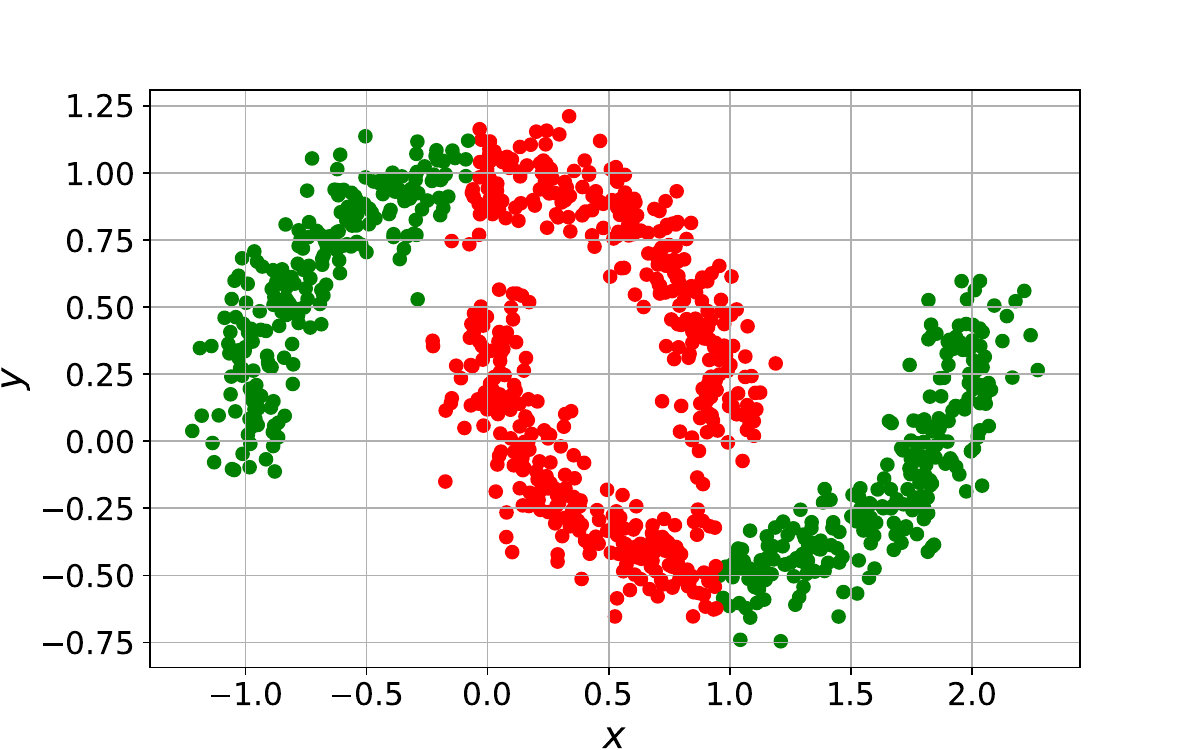}
        \caption{$\delta=\delta_\star-0.002$.}
    \end{subfigure}
    \caption{Two Moons dataset. The green region is that where the lower bound holds, while the red region is that where the lower bound may not hold.}
    \label{fig:tmdlb}
\end{figure}
\begin{figure}[ht]
    \begin{subfigure}{0.32\textwidth}
        \centering
        \includegraphics[width=\textwidth]{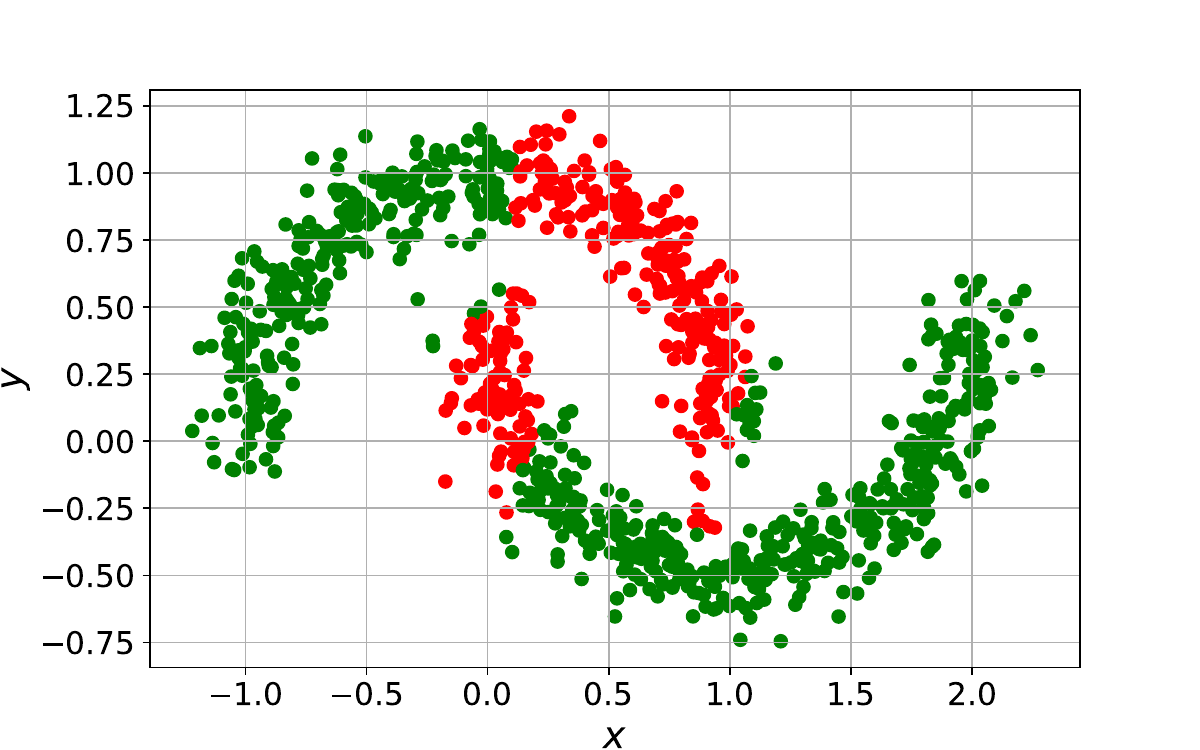}
        \caption{$\delta=\delta_\star-3$.}
    \end{subfigure}
    \hfill
    \begin{subfigure}{0.32\textwidth}
        \centering
        \includegraphics[width=\textwidth]{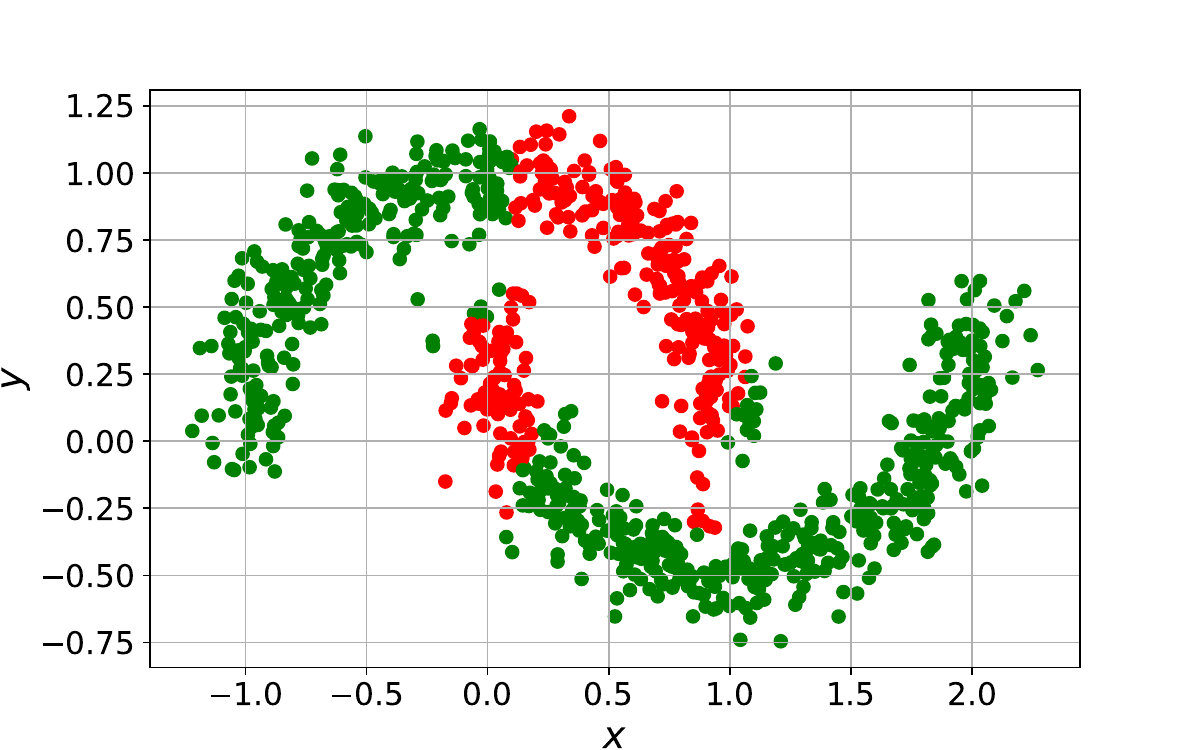}
        \caption{$\delta=\delta_\star-2$.}
    \end{subfigure}
    \hfill
    \begin{subfigure}{0.32\textwidth}
        \centering
        \includegraphics[width=\textwidth]{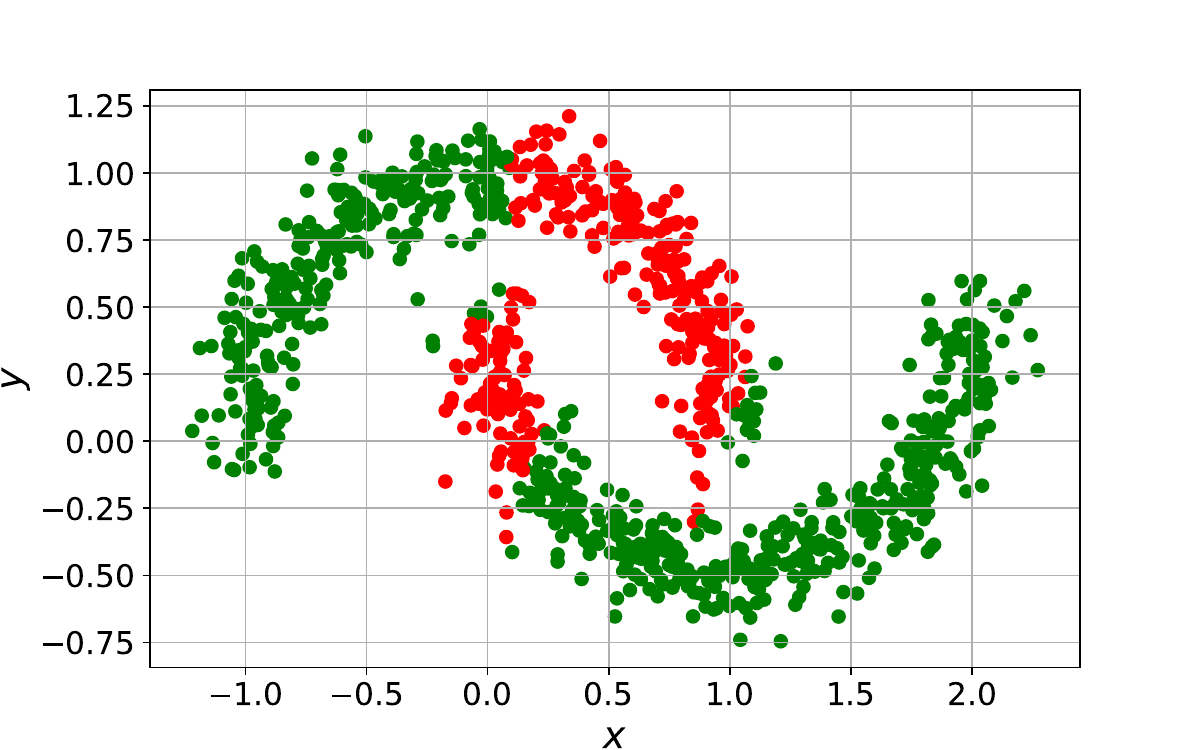}
        \caption{$\delta=\delta_\star-1$.}
    \end{subfigure}
    \caption{Two Moons dataset. The green region is that where the lower bound holds, while the red region is that where the lower bound may not hold.}
    \label{fig:tmdlb1}
\end{figure}
\begin{figure}[ht]
    \begin{subfigure}{0.32\textwidth}
        \centering
        \includegraphics[width=\textwidth]{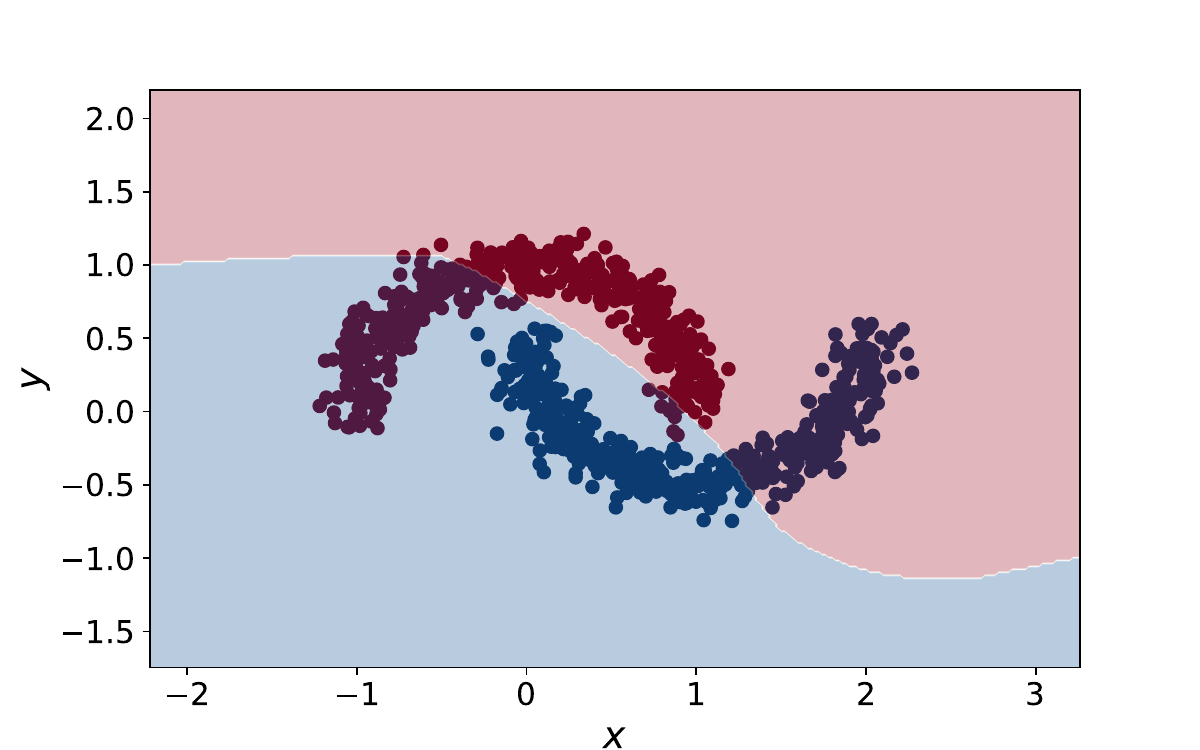}
        \caption{$\delta=\delta_\star-3$. Accuracy: 58.10\%.}
    \end{subfigure}
    \hfill
    \begin{subfigure}{0.32\textwidth}
        \centering
        \includegraphics[width=\textwidth]{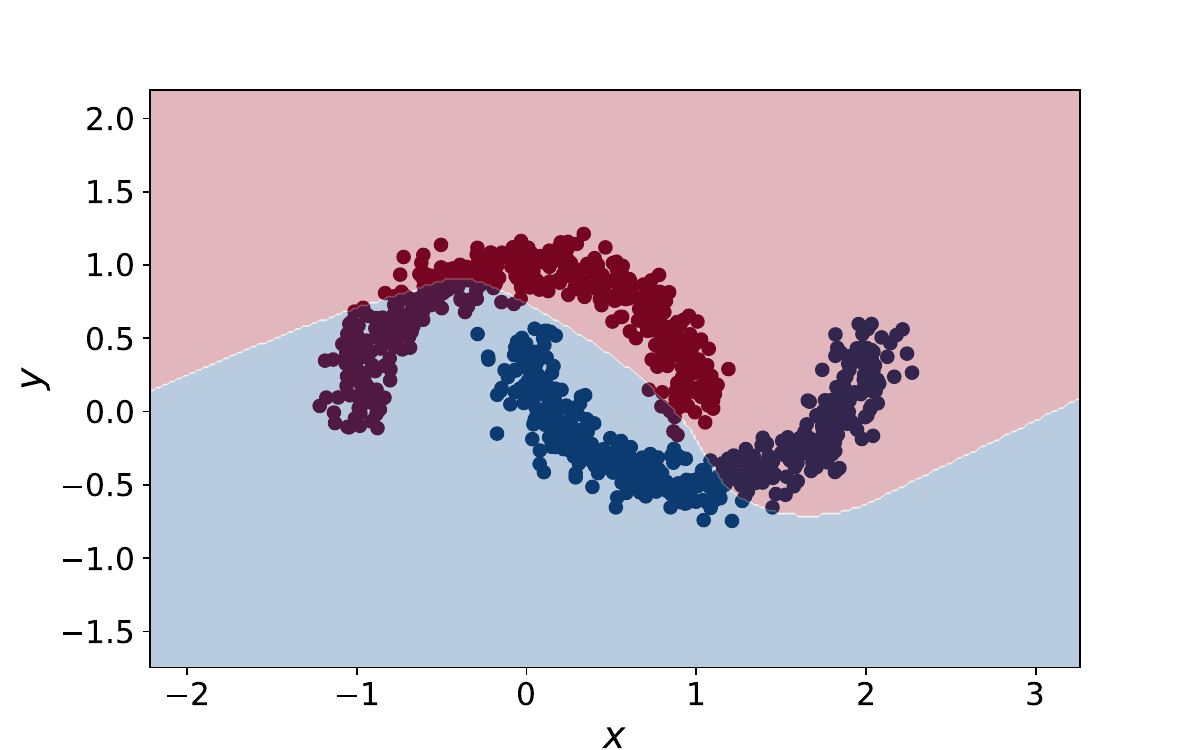}
        \caption{$\delta=\delta_\star-2$. Accuracy: 60.20\%.}
    \end{subfigure}
    \hfill
    \begin{subfigure}{0.32\textwidth}
        \centering
        \includegraphics[width=\textwidth]{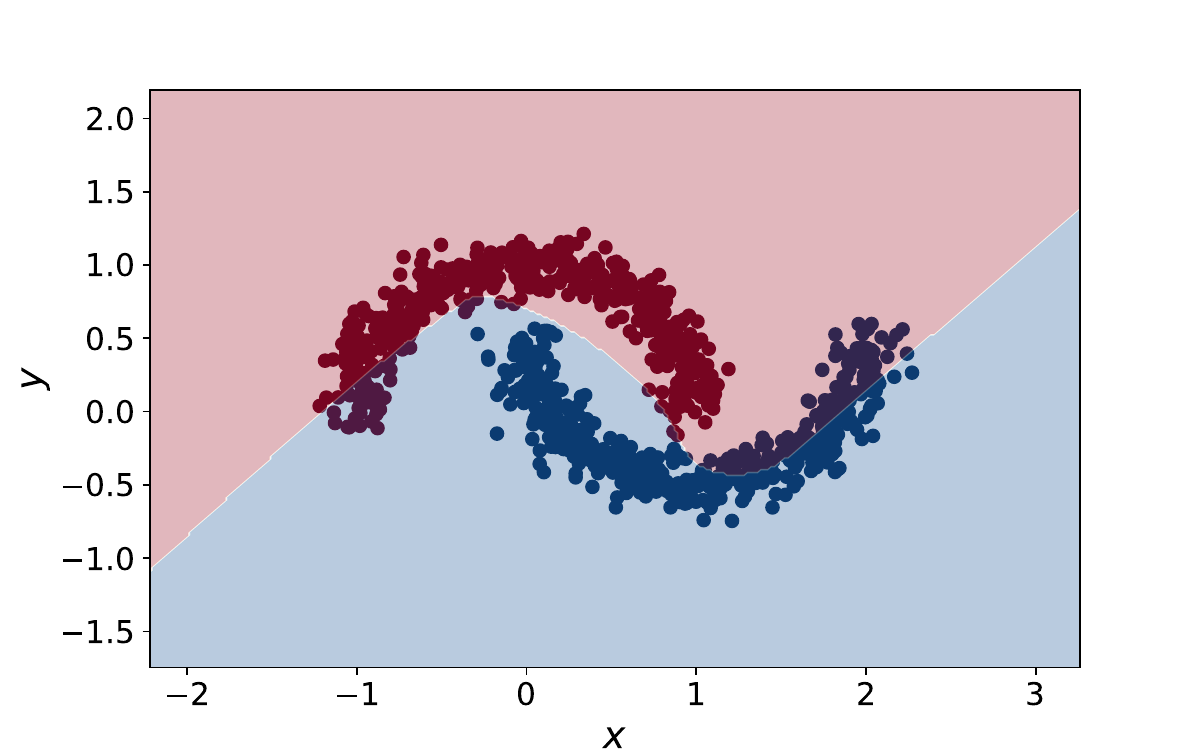}
        \caption{$\delta=\delta_\star-1$. Accuracy: 82.10\%.}
    \end{subfigure}
    \caption{Two Moons dataset. Different classification for different values of $\delta$.}
    \label{fig:tmdacc1}
\end{figure}

\section{Universal approximation theorem with Lipschitz or fixed-norm constraints}\label{sec:uat_unc}
In this section, we investigate how either limiting the Lipschitz constant of the flow map $\phi$ or adding the constraints $\|A_1\|_2\ge1$ a fixed constant and $\|A_2\|_2=1$ affects the universal approximation property of the shallow neural network \eqref{eq:snn}. The intuition tells us that the UAP should hold in both cases. One can think that the restricted expansivity of either the linear layers or the flow map could be counteracted by the unconstrained behaviour of the other part of the network.

In this case, the Lipschitz constant of $\phi\in\mathcal{F}_d$ is restricted, and formalising the universal approximation theorem does not require additional work. Indeed, what is important based on Theorem \ref{th:univ_approx_1} is that the flow map is not a polynomial. Thus, since there are Lipschitz-constrained non-polynomial flow maps, the universal approximation property also persists in this constrained regime. As an example, we can think of the differential equation
\[
\dot{u}(t) = \delta\sigma(u(t) + \lambda),
\]
where $\lambda,\delta\in\R$, $\sigma(x)=\max\{x,\alpha x\}$. The Lipschitz constant of its time$-1$ flow map $\phi$ can be upper bounded by $\exp{\delta \alpha}$, which can be made equal to any positive number. Thus, since $\sigma$ is not linear, the flow map $\phi$ is not a polynomial while still being of constrained Lipschitz constant. This proves the universal approximation theorem in this case.

We now move to the case $\phi\in\mathcal{F}_d$ is unconstrained while the weights of the linear layers are of fixed spectral norm, in particular $\|A_1\|_2\ge1$ and $\|A_2\|_2=1$. We recall that the activation function $\phi$ of \eqref{eq:snn} is the flow map of the neural ODE \eqref{eq:nODE}, depending on the matrix $A\in\R^{d\times d}$ and the vector $b\in\R^d$. If the matrices $A_1$ and $A_2$ are constrained to have fixed 2-norm, then we expect the matrix $A\in\R^{d\times d}$ and the vector $b\in\R^d$ to differ from the ones in the unconstrained case, for the flow map $\phi$ to have Lipschitz constant larger than the one of the flow map $\phi$ in the unconstrained case. In other words, the magnitude of the matrices $A_1$ and $A_2$ can be assimilated in the flow map $\phi$, with a consequent increase of its Lipschitz constant. See \eqref{eq:bound2}.

We formalise this intuition. % by starting with a definition.
%\begin{definition}
%    Fix $p, q \in \N$. We define the space of unit norm affine transformations from $\R^p$ to $\R^q$ as
%    \[
%        \Ac_1(\R^p,\R^q) = \{x\in\R^p \to A x + b \in \R^q \ : \ A\in\R^{q\times p},\ \|A\|_2=1,\ b\in\R^{q} \}.
%    \]
%\end{definition}
We recall that the shallow neural network \eqref{eq:snn} is
\begin{equation*}
    \varphi(x) = A_2 \phi (A_1 x + b_1) + b_2, \quad x\in\R^m,
\end{equation*}
with $A_1\in\R^{d\times m}$, $b_1\in\R^d$, $A_2\in\R^{n\times d}$, $b_2\in\R^n$, and the activation function $\phi:\R^d\to\R^d$ is the flow map of the neural ODE
\begin{equation*}
    \dot{u}(t) = \sigma( A u(t) + b ), \quad t\in[0,1],
\end{equation*}
at time 1, with $A\in\R^{d\times d}$, $b\in\R^d$, $u(t)\in\R^d$ for all $t\in[0,1]$ and $\sigma:\R\to\R$, satisfying Assumption \ref{ass:1}, is applied entrywise.

We consider the piecewise autonomous neural ODE
\begin{equation}\label{eq:nanODE}
    \dot{u}(t) = f(t,u(t)) =
    \begin{cases}
        f_1(u(t)) = u(t), & t\in[t_0,0), \\
        f_2(u(t)) = \sigma( A u(t) + b ), & t\in[0,1), \\
        f_3(u(t)) = u(t), & t\in[1,T]
    \end{cases}
\end{equation}
where $t_0<0$, $T>1$, each component $f_i(u(t))$, $i=1,2,3$, is an autonomous vector field, and $f_2(u(t))$ is the same vector field of \eqref{eq:nODE}. If $u_0\in\R^d$ is an initial datum, then the flow map $\psi:\R^d\to\R^d$ of \eqref{eq:nanODE} at $T$ is
\[
    \psi(u_0) = e^{T-1} \circ \phi \circ e^{-t_0} (u_0) =  e^{T-1} \phi ( e^{-t_0} u_0 ), 
\]
where $\phi:\R^d\to\R^d$ is the flow map of the neural ODE \eqref{eq:nODE} at time 1. So $\psi$ is essentially a scaled $\phi$ both in the input and in the output, with scaling tuned by $t_0$ and $T$. By choosing the appropriate scaling, the magnitude of the matrices $A_1$ and $A_2$ can be assimilated in the flow map, and thus the universal approximation property can be proven also with fixed norm weight matrices.

In this section we then work with the shallow neural network
\begin{equation}\label{eq:nsnn}
    \varphi(x) = A_2 \psi (A_1 x + b_1) + b_2, \quad x\in\R^m,
\end{equation}
with $A_1\in\R^{d\times m}$, $b_1\in\R^d$, $A_2\in\R^{n\times d}$, $b_2\in\R^n$, and the activation function $\psi:\R^d\to\R^d$ is the flow map of the neural ODE \eqref{eq:nanODE}, i.e.
\[
    \varphi(x) = A_2 e^{T-1} \phi (e^{-t_0}(A_1 x + b_1)) + b_2, \quad x\in\R^m.
\]

We need the following definitions.

\begin{definition}
    Let $\psi_{A,b,t_0,T}:\R^d\to\R^d$ be the flow map of the neural ODE \eqref{eq:nanODE} for fixed $A\in\R^{d\times d}$, $b\in\R^d$, $t_0<0$, and $T>1$. We define the space of flow maps of neural ODE \eqref{eq:nanODE}
    \[
        \mathcal{F}_{d,1} = \{ \psi_{A,b,t_0,T} \ : \ A\in\R^{d\times d}, \ b\in\R^d, \ t_0<0,\ T>1 \}
    \]
    for all weight matrices $A\in\R^{d\times d}$, biases $b\in\R^d$, initial time $t_0<0$ and final time $T>1$.
\end{definition}

\begin{definition}
    We define the space of shallow neural networks \eqref{eq:nsnn} from $\R^m$ to $\R^n$
    \[
        \mathcal{H}_{c,1}^\star = \{ \ell_2 \circ \psi \circ \ell_1 \ : \ d\in\N,\ \ell_1\in\Ac_c(\R^m,\R^d),\ \psi\in\mathcal{F}_{d,1},\ \ell_2\in\Ac_1(\R^d,\R^n) \}
    \]
    with activation function defined as the flow map of the neural ODE \eqref{eq:nanODE}.
\end{definition}

We remark that $\mathcal{H}_{c,1}^\star$ is defined as $\mathcal{H}_{c,1}$, where the flow map belongs to $\mathcal{F}_{d,1}$ instead of $\mathcal{F}_d$. We are ready to state the last main result of this paper.

\begin{theorem}\label{th:univ_approx_2}
    The space of functions 
    \[
        \mathcal{H}_{c,1}^\star = \{ \ell_2 \circ \psi \circ \ell_1 \ : \ d\in\N,\ \ell_1\in\Ac_c(\R^m,\R^d),\ \psi\in\mathcal{F}_{d,1},\ \ell_2\in\Ac_1(\R^d,\R^n) \}
    \]
    is a universal approximator for $\cC(\R^{m},\R^{n})$ under the compact convergence topology.
\end{theorem}

\begin{proof}
    Let $f\in \cC(\R^{m},\R^{n})$, $\varepsilon>0$, and $K$ a compact subset of $\R^m$. By Theorem \ref{th:univ_approx_1}, there exists $\varphi\in\mathcal{H}$ such that
    \[
        \|f-\varphi\|_{\infty,K} \le \varepsilon.
    \]
    
    By definition of $\mathcal{H}$, there exists $\ell_1\in\Ac(\R^m,\R^d)$, $\phi\in\mathcal{F}_d$ and $\ell_2\in\Ac(\R^d,\R^n)$ such that
    \[
        \varphi(x) = \ell_2\circ\phi\circ\ell_1(x),
    \]
    i.e. there exist $A_1\in\R^{d\times m}$, $b_1\in\R^d$, $A\in\R^{d\times d}$, $b\in\R^d$, $A_2\in\R^{n\times d}$, $b_2\in\R^n$ such that
    \[
        \varphi(x) = A_2\phi(A_1 x + b_1) + b_2,
    \]
    where $\phi$ is the flow map of the neural ODE \eqref{eq:nODE}. We can rewrite $\varphi$ as
    \begin{align*}
        \varphi(x) &= \frac{A_2}{\|A_2\|_2}\|A_2\|_2 \phi \left( \frac{\|A_1\|_2}{c} \left( c\frac{A_1}{\|A_1\|_2} x + c\frac{b_1}{\|A_1\|_2} \right) \right) + b_2 \\
        &= \hat{A}_2 \|A_2\|_2 \phi \left( \frac{\|A_1\|_2}{c} \left( \hat{A}_1 x + \hat{b}_1 \right) \right) + \hat{b}_2,
    \end{align*}
    with $\hat{A}_1 := c A_1 / \|A_1\|_2$, $\hat{b}_1 := c b_1 / \|A_1\|_2$, $\hat{A}_2 := A_2 / \|A_2\|_2$ and $\hat{b}_2 := b_2$.

    Then, we determine $t_0<0$ and $T>1$ such that
    \[
        e^{-t_0} = \frac{\|A_1\|_2}{c} \quad \text{and} \quad e^{T-1} = \|A_2\|_2,
    \]
    which are
    \[
        t_0 = -\log\frac{\|A_1\|_2}{c} \quad \text{and} \quad T = \log\|A_2\|_2 + 1.
    \]
    Therefore, there exists a neural ODE \eqref{eq:nanODE} with $t_0$ and $T$ as above such that its flow map $\psi$ is
    \[
        \psi = \|A_2\|_2 \circ \phi \circ \|A_1\|_2 = e^{T-1} \circ \phi \circ e^{-t_0},
    \]
    and $\psi\in\mathcal{F}_{d,1}$.
    
    Eventually, if we define
    \[
        \hat{\ell}_1(u) = \hat{A}_1 u + \hat{b}_1 \quad \text{and} \quad \hat{\ell}_2(u) = \hat{A}_2 u + \hat{b}_2,
    \]
    we have that $\hat{\ell}_1\in\Ac_c(\R^m,\R^d)$ and $\hat{\ell}_2\in\Ac_1(\R^d,\R^n)$ and
    \[
        \varphi = \ell_2 \circ \phi \circ \ell_1 = \hat{\ell}_2 \circ \psi \circ \hat{\ell}_1 \in \mathcal{H}_{c,1}^\star.
    \]
    Therefore also $\mathcal{H}_{c,1}^\star$ is a universal approximator for $\cC(\R^{m},\R^{n})$ under the compact convergence topology.
\end{proof}

\begin{remark}\label{rem:comment}
    Introducing neural ODE \eqref{eq:nanODE} is only a theoretical expedient to prove the universal approximation property with fixed norm constraint on $A_1$ and $A_2$. A practical reason the implementation of neural ODE \eqref{eq:nanODE} is not possible is the lack of knowledge a priori of $t_0$ and $T$. In practise it happens that the Lipschitz constant of the flow map $\phi$ of the neural ODE \eqref{eq:nODE} with constrained $A_1$ and $A_2$ is larger than the Lipschitz constant of the flow map $\phi$ with unconstrained $A_1$ and $A_2$. One may think that the magnitudes of $A_1$ and $A_2$ are assimilated by the Lipschitz constant of the flow map $\phi$.
\end{remark}

\section{Conclusions}
We have studied the approximation properties of neural ordinary differential equations (neural ODEs) in the space of continuous functions. Since a neural ODE requires input and output dimensions to be the same, while input and output dimensions of a continuous function are generally different, we have needed to embed an input into the latent space of the neural ODE, and to project the output of the neural ODE into the output space. By composing the neural ODE flow map with such embedding and projection operations, we have got a shallow neural network whose activation function is defined as the flow map of the neural ODE at the final time of the integration interval. Thus, the study of the approximation properties of neural ODEs has led to the study of the approximation properties of shallow neural networks with a particular choice of activation function. We have proven the universal approximation property (UAP) of such shallow neural networks in the space of continuous functions. Furthermore, we have investigated the approximation properties of shallow neural networks whose parameters satisfy specific constraints. In particular, we have constrained the Lipschitz constant of the neural ODE's flow map and the norms of the weights to increase the network's stability. We have proven that the UAP holds if we consider either constraint independently. When both are enforced, there is a loss of expressiveness, and we have derived approximation bounds that quantify how accurately such a constrained network can approximate a continuous function.

\section*{Acknowledgements and Disclosure of Funding}
N.G.\ acknowledges that his research was supported by funds from the Italian MUR (Ministero dell'Universit\`a e della Ricerca) within the PRIN 2022 Project ``Advanced numerical methods for time dependent parametric partial differential equations with applications'' and the PRO3 joint project entitled ``Calcolo scientifico per le scienze naturali, sociali e applicazioni: sviluppo metodologico e tecnologico''. N.G. and F.T. acknowledge support from MUR-PRO3 grant STANDS and PRIN-PNRR grant FIN4GEO. A.D.M., N.G. and F.T. are members of the INdAM-GNCS (Gruppo Nazionale di Calcolo Scientifico). E.C. and B.O. acknowledge that their research was supported by Horizon Europe and MSCA-SE project 101131557 (REMODEL). D.M. acknowledges support from the EPSRC programme grant in ``The Mathematics of Deep Learning'', under the project EP/V026259/1, the Department of Mathematical Sciences of NTNU, and the Trond Mohn Foundation for the support during this project.

%\bibliographystyle{plain}
%\bibliography{biblio.bib}

\end{document}